\documentclass[a4paper]{amsart}

\usepackage[foot]{amsaddr}
\usepackage{lipsum}
\usepackage{amsfonts,amssymb}
\usepackage{graphicx}
\usepackage{xcolor}
\usepackage{epstopdf}
\usepackage{algorithmic}
\ifpdf
  \DeclareGraphicsExtensions{.eps,.pdf,.png,.jpg}
\else
  \DeclareGraphicsExtensions{.eps}
\fi

\title{Hyperbolic Relaxation of  $k$-Locally Positive Semidefinite Matrices}

\author{Grigoriy Blekherman$^*$}
\address[*]{School of Mathematics, Georgia Institute of Technology}
\email[*]{greg@math.gatech.edu}
\author{Santanu S. Dey$^{**}$}
\address[**]{School of Industrial and Systems Engineering, Georgia Institute of Technology}
\email[**]{santanu.dey@isye.gatech.edu}
\author{Kevin Shu$^*$}
\email[***]{kshu8@gatech.edu}
\author{Shengding Sun$^*$}
\email[****]{ssun313@gatech.edu}
\thanks{Grigoriy Blekherman and Kevin Shu were partially supported by NSF grant DMS-1901950. Santanu S. Dey was supported by ONR under grant No. N4458-NV-ONR. Shengding Sun was supported by ARC-TRIAD Fellowship.}

\usepackage{amsopn}

\newcommand{\R}{\mathbb{R}}

\DeclareMathOperator{\Sym}{Sym}

\newcommand{\change}[1]{#1}
\DeclareMathOperator{\diag}{diag}
\newcommand{\tr}{\textup{Trace}}

\newtheorem{theorem}{Theorem}
\newtheorem{corollary}{Corollary}
\newtheorem{lemma}{Lemma}
\newtheorem{definition}{Definition}
\newtheorem{observation}{Observation}
\newtheorem{remark}{Remark}
\newtheorem{conjecture}{Conjecture}
\numberwithin{theorem}{section}

\begin{document}

% REQUIRED
\begin{abstract}
\change{A successful computational approach for solving large-scale positive semidefinite (PSD) programs is to enforce PSD-ness on only a collection of submatrices}. \change{For our study, we let $\mathcal{S}^{n,k}$ be the convex cone of $n\times n$ symmetric matrices where all $k\times k$ principal submatrices are PSD}. We call a matrix in this $k$-\emph{locally PSD}.  In order to compare \change{$S^{n,k}$} to the of PSD matrices, we study %the set of possible 
eigenvalues of $k$-{locally PSD} matrices.
The key insight in this paper is that \change{there is a convex cone $H(e_k^n)$ so that if $X \in \mathcal{S}^{n,k}$, then the vector of eigenvalues of $X$ is contained in $H(e_k^n)$. } The cone \change{$H(e_k^n)$ is} the hyperbolicity cone of the elementary symmetric polynomial $e^k_n$ (where $e_k^n(x) = \sum_{S \subseteq [n] : |S| = k} \prod_{i \in S} x_i$) with respect to the all ones vector. Using this insight, we are able to improve previously known upper bounds on the Frobenius distance between matrices in $\mathcal{S}^{n,k}$ and PSD matrices. We also study the quality of the convex relaxation $H(e^n_k)$. We first show that this relaxation is tight for the case of $k = n -1$, that is, for every vector in $H(e^n_{n -1})$ there exists a matrix in $\mathcal{S}^{n, n -1}$ whose eigenvalues are equal to the components of the vector. We then prove a structure theorem on nonsingular matrices in $\mathcal{S}^{n,k}$ all of whose $k\times k$ principal minors are zero, which we believe is of independent interest.
%We then prove a structure theorem that precisely characterizes the non-singular matrices in $\mathcal{S}^{n,k}$ whose vector of eigenvalues belongs to the boundary of $H(e^n_k)$. 
This result shows shows that for $1< k < n -1$ ``large parts" of the boundary of $H(e_k^n)$ do not intersect with the eigenvalues of matrices in $\mathcal{S}^{n,k}$. 
\end{abstract}

\maketitle


\section{Introduction}
\subsection{$k$-Locally positive semidefinite matrices}

Positive semidefinite (PSD) matrices are of fundamental interest in a wide variety of fields, ranging from optimization~\cite{wolkowicz2012handbook} to physics~\cite{cavalcanti2016quantum}.
Formally, a symmetric matrix $X\in \Sym_n$ is PSD if and only if \[u^{\top}Xu \geq 0 \,\,\text{ for all }\,\, u \in \mathbb{R}^n.\]
The property of being positive semidefinite is very strong, and implies a large amount of structure in a matrix. For example, all the eigenvalues of a  PSD matrix are non-negative. Another important property of a PSD matrix is that all its principal submatrices are also PSD. There are various conceivable converses to this fact which fail to hold; for instance, even if all of the proper submatrices of a matrix are PSD, it is still possible for the matrix to have a negative eigenvalue. Nevertheless, one might be interested in a partial converse: if enough submatrices of a matrix are PSD, we should expect that the matrix is `close' to being PSD, in some sense. Such a result would help explain a phenomenon observed in various recent computational experiments, where the constraint of a matrix being PSD is relaxed to that of some principal submatrices being PSD. It has been empirically observed that the resulting relaxation has an optimal objective function value close to the optimal objective function of the original problem. See~\cite{kocuk2016strong,kocuk2017matrix,sojoudi2014exactness} for examples involving SDP relaxation of optimal electrical power flow problem and~\cite{baltean2018selecting,DeyWorking,qualizza2012linear} for examples involving SDP relaxation of box quadratic programs.

To formally understand the relaxation of enforcing positive semidefiniteness on submatrices, we investigate a class of matrices, where we  impose the conditions that all $k\times k$ principal submatrices of an $n\times n$ matrix are PSD. We will call such a matrix $k$-\emph{locally PSD}. This terminology is meant to suggest that we only check the PSD conditions locally on some small parts of the matrix rather than globally. Let 
\begin{align}
\mathcal{S}^{n,k}:= \{X \in \Sym_n\,|\, \textup{every } k \times k \textup{ principal submatrix of }X \textup{ is PSD} \}\label{def:s_n_k}
\end{align}
be the set of $k$-locally PSD matrices.  The set $\mathcal{S}^{n,k}$ is a closed  convex cone and its dual cone is the set of symmetric matrices with factor width $k$, defined and studied in~\cite{boman2005factor,Permenter2018,gouveia2019sums}. The set of symmetric matrices with factor width 2 is the set of scaled diagonally dominant matrices~\cite{boman2005factor,wang2019polyhedral}, \change{i.e.,} symmetric matrices $A$ such that $DAD$ is diagonally dominant for some positive diagonal matrix $D$. The paper~\cite{ahmadi2019dsos} uses scaled diagonally dominant matrices for constructing inner approximation of the \change{PSD} 
%\gb{PSD?} 
\change{cone} for use in solving polynomial optimization problems. See~\cite{hurwitz1891ueber, reznick1987quantitative,reznick1989forms} for related papers.

\change{For $X \in \Sym_n$, we let $\lambda(X) = (\lambda_1(X), \dots, \lambda_n(X))$, where $\lambda_1(X) \le \lambda_2 (X)\le \dots \lambda_n(X)$ are the eigenvalues of $X$ counting multiplicity. We say that a vector $\lambda \in \R^n$ is a vector of eigenvalues of $X$ if it can be obtained from $\lambda(X)$ by permuting its coordinates.}

\emph{Our main goal is to understand properties of eigenvalues of $k$-locally PSD matrices. } In particular, we would like to understand how as $k$  gets closer to $n$, the matrices in $\mathcal{S}^{n,k}$ become closer to PSD matrices in terms of eigenvalues.  This work extends (and improves) results in a recent paper~\cite{blekherman2020sparse}, by identifying a relaxation of the set of eigenvalues of $k$-locally PSD matrices.  This relaxation is based on the machinery of hyperbolic polynomials and hyperbolicity cones, which we discuss next.

\subsection{The \change{Hyperbolic  Relaxation}}
In order to motivate the the machinery of hyperbolic polynomials and corresponding hyperbolicity cones (which we formally define later), let us first construct a natural relaxation of the set of eigenvalues of matrices in $\mathcal{S}^{n,k}$. 

\change{
Given an $n\times n$ matrix $X$, recall the definition of the characteristic polynomial of $X$:
\[
p_{X}(t) = \det(X - tI) = \sum_{{\ell}=0}^n (-1)^{n-\ell}c_{\ell}^{n}(X)t^{n-\ell},
\]
where 
\[
c_{\ell}^{n}(X) = 
\sum_{\genfrac{}{}{0pt}{2}{S \subseteq [n]}{|S| = {\ell}}} \det(X|_S).
\]

Here, $X|_S$ is the principal submatrix of $X$ obtained by restricting $X$ to the rows and columns contained in $S$.
Notice that if $X \in S^{n,k}$, then for all $S \subseteq [n]$ with $|S| \le k$, $X|_S$ is PSD and in particular, then $\det(X|_S) \ge 0$. This implies that $c_{\ell}^n(X) \ge 0$ for $\ell \le k$.

Let us introduce the set
\[
H(c_k^n) = \{X \in \Sym_n : \forall \ell \le k,\;c_{\ell}^n(X) \ge 0\}.
\]
Our first observation is then that:
\begin{eqnarray}\label{eq:strawman}
\mathcal{S}^{n,k} \subseteq H(c_k^n).
\end{eqnarray}

The roots of $p_X$ are precisely the negatives of the eigenvalues of $X$, so we have that 
\[
p_X(t) = \prod_{i=1}^n (\lambda_i - t) = \sum_{\ell=0}^n (-1)^{n-\ell}e_k^n(\lambda_1, \dots, \lambda_n)t^{n-\ell},
\]
where $\lambda_1, \lambda_2, \dots \lambda_n$ are the eigenvalues of $X$ in any order, counting multiplicity, and $e^n_ k \in \mathbb{R}[x_1, \dots, x_n]$ is the elementary symmetric polynomial $$e_k^n(x) = \sum_{\genfrac{}{}{0pt}{2}{S \subseteq [n]}{|S| = {\ell}}}\prod_{i \in S} x_i.$$

Comparing coefficients of the $t^k$ terms, we see that for $0 \le k \le n$,
\[
c_k^n(X)= e_k^n(\lambda_1, \dots, \lambda_n).
\]

Combining our previous observations, we see that $X \in H(c_k^n)$ if and only if, $e_{\ell}^n(\lambda_1, \dots, \lambda_n) \ge 0$ for $\ell \le k$.

We will define the set
\[
H(e_k^n) = \{\lambda \in \R^n : \forall \ell \le k,\;e_{\ell}^n(\lambda) \ge 0\}.
\]

Combining these observations, we obtain the following:
\begin{observation}
$\mathcal{S}^{n,k} \subseteq H(c_k^n)$. Also, $\lambda = (\lambda_1, \dots, \lambda_n)$ is a vector of eigenvalues of some $X \in H(c_k^n)$ if only if $\lambda \in H(e_k^n)$. 
%$\mathcal{S}^{n,k} \subseteq H(c_k^n)$, and if $X \in H(c_k^n)$, and $\lambda = (\lambda_1, \dots, \lambda_n)$ is a vector of eigenvalues of $X$, then $\lambda \in H(e_k^n)$ 
%\gb{Shouldn't the second part be an if and only if?}.
\end{observation}

The set $H(e_k^n)$ will turn out to be the hyperbolicity cone of the polynomial $e_k^n$ with respect to the all ones vector, and in particular will turn out to be invariant to permutations of the coordinates and convex~\cite{gaarding1951linear,guler1997hyperbolic}. We will refer to $H(e_k^n)$ as the \emph{hyperbolic relaxation} for the eigenvalues of $\mathcal{S}^{n,k}$ Similarly, $H(c_k^n)$ will be the hyperbolicity cone of $c_k^n$ with respect to the identity matrix. The cone $H(c_k^n)$ is well-known in the literature and is sometimes referred to as the $(n-k)^{th}$ Renegar derivative of the PSD cone~\cite{RenegarDer06}.%TODO: Cite something

By exploiting properties of these convex cones, we obtain bounds that indicate that if $X\in \mathcal{S}^{n,k}$, then in fact, $X$ is close to being PSD in a number of different norms.

As an example, notice that $e_n^n = \prod_{i=1}^n x_i$, and $H(e_n^n) = \R^n_+$, the nonnegative orthant. Similarly, $H(c_n^n)$ is the PSD cone, and we have that $H(e_n^n)$ is exactly the set of possible eigenvalue vectors for matrices in $H(c_n^n)$.

}

This observation motivates us to ask the following related questions that tie in with our goal of understanding properties of $\lambda(X)$ for $X \in \mathcal{S}^{n,k}$:
\begin{itemize}
\item \emph{Is it possible to obtain understanding of $\{\lambda(X)\,|\, X \in \mathcal{S}^{n,k}\}$ in comparison to eigenvalues of PSD matrices,  by studying properties of \change{$H(e_k^n)$}?}
\item \emph{How good is the approximation of the set $\{\lambda(X)\,|\, X \in \mathcal{S}^{n,k}\}$ by $H(e_k^n)$?}
\end{itemize}

In this paper, we answer `yes' to the first question by improving on results in~\cite{blekherman2020sparse} via the \change{hyperbolic} relaxation of the set of eigenvalues of matrices in $\mathcal{S}^{n,k}$. This motivates us to delve further into the second question and we verify various structural results that give a better understanding of relationship between the sets $\{\lambda(X)\,|\, X \in \mathcal{S}^{n,k}\}$ and $H(e_k^n)$. One particularly interesting result that we would like to highlight is a structure theorem for matrices in $\mathcal{S}^{n,k}$ all of whose principal $k\times k$ minors vanish (Theorem \ref{thm:struc}). This theorem has relations to previous results in linear algebra: Theorem 13 in \cite{thompson1968principal} and results of \cite{Oeding_2011}.

\subsection{Notation}
For a positive integer $n$, let $[n] := \{1, \dots, n\}$.  Let $\Sym_n$ denote the vector space of $n \times n$ symmetric matrices. Let $\mathcal{S}^n$ denote the cone of PSD matrices inside of $\Sym_n$. Note that $\mathcal{S}^{n,n} = \mathcal{S}^n$. If $M \in \mathcal{S}^n$, we write $M \succeq 0$. \change{We use $S^{n,k}$ to refer to the $k$-locally PSD matrices, which are defined above in equation \eqref{def:s_n_k}}. We will refer to $k$-locally PSD matrices for convenience as locally PSD matrices if $k$ is clear from context.
%For $S \subseteq [n]$, let $M|_S$ denote the submatrix of $M$ obtained by removing the rows and columns that are not in $S$.
%Let $\mathcal{S}^{n,k}$ denote the set of symmetric PSD matrices with the property that for all $S \subseteq [n]$ with $|S| = k$, we have $M|_S \succeq 0$. In this definition, $\mathcal{S}^{n} = \mathcal{S}^{n,n}$. We will refer to these as $k$-PSD matrices, or simply locally PSD matrices if $n$ and $k$ are clear from context.
An important example of a non-PSD matrix lying in $\mathcal{S}^{n,k}$ is given by
\begin{eqnarray}\label{eqn:g_n_k}
    G(n, k) = \frac{k}{k-1}I - \frac{1}{k-1} \vec{1} \vec{1}^{\top}.
\end{eqnarray}
\change{Here,} $\vec{1}$ denotes the all ones vector of dimension $n$. All diagonal entries of $G(n,k)$ are identically 1, and all off-diagonal entries are identically $-\frac{1}{k-1}$. Notice that all $k\times k$ \change{principal} minors of $G(n,k)$ vanish, but the matrix is nonsingular.

Given a matrix $M \in \mathcal{S}^{n,k}$ and a diagonal matrix $D$ with non-zero diagonal entries, observe that
\[
    DMD \in \mathcal{S}^{n,k}.
\]
We say that the matrix $ DMD$ is \emph{diagonally congruent} to $M$. By applying Sylvester's law of inertia to any submarix of $M$, the number of positive and negative eigenvalues are conserved for the same submatrix of a diagonally  congruent matrix. In particular, a principal submatrix of a diagonally congruent matrix is singular iff the same submatrix is singular in the original matrix.

%\change{For $X \in \Sym_n$, we let $\lambda(X) = (\lambda_1(X), \dots, \lambda_n(X))$, where $\lambda_1(X) \le \lambda_2 (X)\le \dots \lambda_n(X)$ are the eigenvalues of $X$ counting multiplicity. We say that a vector $\lambda \in \R^n$ is a vector of eigenvalues of $X$ if it can be obtained from $\lambda(X)$ by permuting its coordinates.} 
%\gb{Shiould we move this earlier and use it in our preceding discussion?}

The rest of the paper is organized as follows: Section~\ref{sec:main} lists all our main results and Section~\ref{sec:conclusion} concludes with some open questions. Then Section~\ref{sec:pre} presents background results needed for proving the main results. The remaining sections present our proofs of the main results.

\section{Main results}\label{sec:main}
\subsection{Bounds on minimum eigenvalues of matrices in $\mathcal{S}^{n,k}$}
The primary way we can measure the distance between a matrix in $\mathcal{S}^{n,k}$ and the cone of PSD matrices is by considering the smallest eigenvalue of such a matrix. Certainly, if the minimum eigenvalue of a matrix is nonnegative, then the matrix is positive semidefinite, and we will say that a (suitably normalized) matrix is close to being PSD if its minimum eigenvalue is close to being nonnegative. We show that if $k$ is sufficiently close to $n$, then the $k$-locally PSD matrices are close to the PSD matrices. 
%Precisely, every $k$-locally PSD matrix with trace 1 has minimum eigenvalue at least $\frac{k-n}{(k-1)n}$. The proof of this fact is itself very short, and uses the theory of real stable polynomials in an interesting way. Specifically, we will show that the set of eigenvalues of the class of $k$-locally PSD matrices lie in the \textbf{hyperbolicity cone} of an elementary symmetric polynomial, which is to be defined. The hyperbolicity cone will be defined below, but critically, it is basis invariant, making it easier to control the eigenvalues of matrices in this cone.
%
%Once we have this relaxation, we will also investigate how close this is to the true set of eigenvalues of the $k$-locally PSD matrices. We will show that the relaxation is exact when $k = n-1$. For all other cases, we then show some structure theorems that show that the eigenvalues of $k$-locally singular matrices are strictly contained in this hyperbolicity cone.
%
%As a final positive note, we note that we can precisely characterize the exact set of eigenvalues in the case of $4\times 4$ 2-locally PSD matrices whose eigenvalues lie on the boundary of the hyperbolicity relaxation. It is exactly the set of possible eigenvalues on the boundary of hyperbolicity relaxation that have at most one negative eigenvalue. We conjecture that a similar result holds for all $n$ and $k$, but so far, the only proof of this fact for $n=4$ and $k=2$ relies heavily on computer calculations that do not generalize well.
Let $\lambda_1(M)$ be the minimum eigenvalue of a matrix $M \in \mathcal{S}^{n,k}$. 
%This minimum eigenvalue measures the distance of $M$ to $\mathcal{S}^n$. 
Because $\lambda_1$ is 1-homogeneous, i.e., if $a \ge 0$, then $\lambda_1(aM) = a \lambda_1(M)$, we should try to compare $\lambda_1(M)$ for $M \in \mathcal{S}^{n,k}$ to other 1-homogeneous (also called positively homogeneous) quantities on $M$. 

%Formally, we will consider a class $\mathcal{F}$ of functions $F:\Sym_n \rightarrow \mathbb{R}$, which are $1$-homogeneous functions  such that $F$ is positive on $\mathcal{S}^{n,k}\setminus \{0\}$, and is basis invariants, i.e., the functions depends entirely of the eigenvalues of $M$.  Let $f:\mathbb{R}^n \rightarrow \mathbb{R}$ such that $f(\lambda(M)) = F(M)$. Then we will also assume $f$ is a symmetric convex function satisfying $f(u) \geq f(u^{-})$ where $u^{-}_i = \textup{min}\{0, u_i\}$.
%Some examples of functions $F \in \mathcal{F}$ are norms like Frobenium norm of $M$ and more generally the Schatten $p$-norms, $|M|_p = \sum_{i=1}^n |\lambda_i(M)|^p$ for $p \geq 0$. Other functions like the trace of $M$ is a non-norm function belonging to $\mathcal{F}$. Our first main result shows that for any $F\in \mathcal{F}$ the smallest minimal eigenvalue out of all matrices in $M\in \mathcal{S}^{n,k}$ normalized to have $F(M)=1$ is acheived by an appropriately rescaled version of $G(n,k)$.

Formally, let $\mathcal{F}$ be the class of functions $F:\Sym_n \rightarrow \mathbb{R}$ so that $F$ is a unitarily invariant matrix norm (thus, a norm depending entirely on
the eigenvalues) or the trace function. Examples of unitarily invariant matrix norms are the Schatten $p$-norms, \change{$\|M\|_p = \sqrt[p]{\sum_{i=1}^n |\lambda_i(M)|^p}$} for $p \geq 1$. \change{Note that the} Frobenius norm is a special case of the Schatten $p$-norm when $p =2$. \change{Also, recall that $G(n,k)$ is defined in equation \eqref{eqn:g_n_k}}

\begin{theorem}\label{thm:min_eigen}
Let $k \in \{2, \dots, n\}$. Let $F \in \mathcal{F}$ and let $\tilde{G}(n,k)=\frac{G(n,k)}{F(G(n,k))}$. For any $M\in \mathcal{S}^{n,k}$ with $F(M)=1$, the \change{minimum} eigenvalue of $M$ is at least as large as the \change{minimum} eigenvalue of $\tilde{G}(n,k)$, that is,
   $$\lambda_1(M) \geq \lambda_1(\tilde{G}(n,k))\quad \text{for all} \quad M \in \mathcal{S}^{n,k}\quad \text{such that} \quad F(M) = 1.$$
The bound on $\lambda_1(M)$ is tight since $\tilde{G}(n,k) \in \mathcal{S}^{n,k}$ achieves this bound.
\end{theorem}

\change{
An immediate corollary of Theorem \ref{thm:min_eigen} in the case when $F$ is the trace function is the following:
\begin{corollary}
Let $k \in \{2, \dots, n\}$. For any $M \in \mathcal{S}^{n,k}$ such that $\tr(M) = 1$, we have
\[
    \lambda_{1}(M) \ge \frac{k-n}{n(k-1)}.
\]
\end{corollary}
}
%{\color{blue}We should list a few consequences of this using some specific functions? Also maybe mention improving theorem from previous paper.}

The proof of Theorem~\ref{thm:min_eigen}  is a direct application of the fact that $H\left(e_k^n \right)$ is a convex relaxation of the set of eigenvalues of matrices in $\mathcal{S}^{n,k}$. This allows us to write a convex relaxation of the optimization problem minimizing $\lambda_1(M)$ over $M \in \mathcal{S}^{n,k}$. The optimal solution of this convex relaxation is the bound obtained in the above \change{theorem}. 
\change{
\begin{remark}
The bound on $\lambda_1(M)$ presented in Theorem~\ref{thm:min_eigen} holds for $M \in H(e^n_k)$. Therefore, this bound can be used to provide upper bounds on the distance between the PSD cone and the Renegar derivative $H(c^n_k)$ of the PSD cone.
\end{remark}
}

%To see an application  of the above result, 
We can use Theorem~\ref{thm:min_eigen} to bound Frobenius distances of matrices in $\mathcal{S}^{n,k}$ from those in $\mathcal{S}^n$ (when we normalize the matrices using the Frobenius norm) as in the following result. 

\begin{corollary}\label{cor:improve} Let $\overline{\textup{dist}}(\mathcal{S}^{n,k},\mathcal{S}^n) =  \textup{max}_{A \in \mathcal{S}^{n,k}: \|A\|_F = 1} \left( \textup{min}_{G \in \mathcal{S}^n} \|A - G\|_F \right)$. Then $\overline{\textup{dist}}(\mathcal{S}^{n,k},\mathcal{S}^n)\leq \frac{(n-k)^{3/2}}{\sqrt{(n-k)^2+(n-1)k^2}}$. 
\end{corollary}

This corollary improves upon Theorem 2 in a previous paper~\cite{blekherman2020sparse}. Our new result has a better constant factor, and applies to all regimes of $k$ and $n$.  
%\ss{The published version of previous paper has denominator $n^{3/2}$, but that was simplified from this denominator. Maybe we can append $\leq (\frac{n-k}{n})^{3/2}$ to this corollary. }

\subsection{Tightness of the relaxation $H\left(e_k^n \right)$ for the set of eigenvalues of matrices in $\mathcal{S}^{n,k}$}
%We analyze the behavior of the hyperbolic relaxation and $\mathcal{S}^{n,k}$ under the eigenvalue map.

%For every $n > k$, we have that the containment $\mathcal{S}^{n,k} \subseteq H(c^n_k)$ is strict. Interestingly though, for $k = n-1$, we this relaxation is tight on the level of eigenvalues in the following sense:

%

%We have seen that the converse of \eqref{eq:keyinsight} \gb{did we remove this reference?} holds when $k = n$. 
\change{We have seen that the vector of eigenvalues of matrices in $S^n_k$ is precisely $H(e^n_k)$ for $n = k$.} \change{We next show that this observation holds for two additional cases of $k$. }

\begin{theorem}\label{thm:k=n-1}
%    If if $M$ is any matrix in $H(c^n_{n-1})$, then there is some matrix $M' \in \mathcal{S}^{n,n-1}$ so that $\lambda(M) = \lambda(M')$.
\change{If $k$ is one of $1,n-1$ or $n$, and $x \in H(e_{k}^n)$,} then $x$ is the vector of eigenvalues of some matrix in $\mathcal{S}^{n,k}$.
%Let $u \in H(e_{n-1}^n)$. Then ther
\end{theorem}
Since \change{$H(e_{k}^n)$} is a convex set \change{for all $k$}, in particular, Theorem~\ref{thm:k=n-1} implies that the set of possible vectors of eigenvalues for matrices in $\mathcal{S}^{n,n-1}$ is convex, although there does not seem to be an easy way to see this directly. 
%\delete{Formally, if $\bar{\lambda}(X)$ is the vector of eigenvalues of $X$ listed in non-decreasing order, then we define the set of all ``eigenvalues vectors"  of $X$ as the set of all the ways to permute the entries of the vector $\bar{\lambda}(X)$.} 
Therefore we have:
\begin{corollary}
    \change{The set of all vectors of eigenvalues of matrices} in $\mathcal{S}^{n,n-1}$ is convex.
\end{corollary}
%{\color{blue} GB:There should be a precise theorem here.}

On the other hand, we show that $\{\lambda(X)\,|\, X \in \mathcal{S}^{n,k}\}$ is strictly contained in $H(e_k^n)$ for $2 < k < n - 1$, \change{and thus completely characterize the tightness of this hyperbolic relaxation.}

In order to do so, we examine the boundary of the cone $H(e_k^n)$, and obtain a result which is of independent interest. \change{Recall that the boundary of $H(e_k^n)$ is precisely the set of points in the hyperbolicity cone on which the polynomial $e_k^n$ vanishes.}%  $\{x \in H(e_k^n) : e_k^n(x) = 0\}$, by Lemma \ref{lem:hyperbolic_boundary}.} \gb{That isn't right. There are more points where where the polynomial vanishes. We just need containment...}

Recall that if $M$ is a matrix, and $\lambda(M)$ is a vector of eigenvalues of $M$, we have that
\[
    e_k^n(\lambda(M)) = \sum_{S \subseteq [n] : |S| = k} \det(M|_S).
\]
Because of this, we see that if $e_k^n(\lambda(M)) = 0$, and $M\in \mathcal{S}^{n,k}$, then we have that $\det(M|_S) = 0$ for all subsets $S$ of size $k$. We formalize this notion:
\begin{definition}
We say that a matrix $M$ is $(n,k)$-{locally singular} if it lies in $\mathcal{S}^{n,k}$ and all of the $k\times k$ minors of $M$ are singular. 
\end{definition}
We see that for $M \in\mathcal{S}^{n,k}$,  $\lambda(M)$ is on the boundary of $H(e_k^n)$ if and only if $M$ is $(n,k)$-locally singular.
Sometimes $(n,k)$ is omitted and we say a matrix is locally singular if $n$ and $k$ are clear from context.

The simplest class of examples of locally singular matrices are matrices of rank less than $k$. In particular, rank 1 PSD matrices will be locally singular for any $n$ and $k$. A more interesting example of locally singular matrices are the $G(n,k)$ matrices \change{defined in equation (\ref{eqn:g_n_k})}. Not only are these matrices locally singular, but they are also non-singular. From this example we can construct an $n$-dimensional space of locally singular matrices by taking an arbitrary invertible diagonal matrix $D$, and considering
\[
    DG(n,k)D \in \mathcal{S}^{n,k},
\]
that is the set of matrices that are \emph{diagonally congruent} to $G(n,k)$. 
% {\color{blue}GB: diagonally congruent?} to $G(n,k)$. 
It follows from Sylvester's law of inertia, applied to the various submatrices of $G(n,k)$, that any matrix diagonally congruent to $G(n,k)$ is in fact locally singular and nonsingular. 

On the other hand, it is worth performing a quick dimension counting heuristic to estimate how many possible matrices satisfy these conditions: Each submatrix that is constrained to be singular imposes a single polynomial equation on the possible solution set. If $n-1 > k > 2$, then we see that the number of equations is $\binom{n}{k}$, which is in fact greater than the dimension of the space of $n\times n$ symmetric matrices, which is $\binom{n}{2}$. This indicates that solutions to this type of system should be somewhat ``uncommon". Therefore, it is perhaps not surprising that in the cases when $n-1 > k > 2$, we are able to show that  all of the locally singular matrices in $\mathcal{S}^{n,k}$, which are not singular, are diagonally congruent to $G(n,k)$.

\change{The following theorem formalizes this idea. Both the statement and the proof of this theorem seem closely related to Theorem 13 in \cite{thompson1968principal}.
More generally, the theorem can be viewed as giving semialgebraic relations between the various principal minors of a symmetric matrix. A complete characterization of the algebraic relations between the principal minors of a symmetric matrix was given in \cite{Oeding_2011}.}

\begin{theorem}\label{thm:struc}
%    If $n-1 > k > 2$ or $n = 4$ and $k = 2$, then any locally singular matrix in $\mathcal{S}^{n,k}$ that is nonsingular is diagonally congruent to $G(n,k)$. Equivalently, 
Let $n-1 > k > 2$ or $(n,k) = (4,2)$. Suppose that $M \in \mathcal{S}^{n,k}$, \change{$M$ is $(n,k)$-locally singular}, and $M$ is invertible. Then $M$ must be diagonally congruent to  $G(n,k)$.
\end{theorem}

Note that this immediately implies that there are points on the boundary of $H(e_k^n)$ which are not the eigenvalues of any matrix in $\mathcal{S}^{n, k}$: since $G(n,k)$ has only one negative eigenvalue, any matrix diagonally congruent to $G(n,k)$ has at most one negative eigenvalue, and there are points on the boundary of $H(e_k^n)$ with as many as $n - k$ negative entries and no zero entries.

\change{When $k=2$ and $n=4$, locally singular matrices in $\mathcal{S}^{4,2}$ are diagonally congruent to a symmetric matrix with diagonal entries identically one and $\pm 1$ off-diagonal entries. We can numerically check that a such matrix has at most one negative eigenvalue. Thus by Cauchy's Interlacing Theorem, for any $n>4$ and $k=2$, any locally singular matrix in $\mathcal{S}^{n,2}$ can have at most $n-3$ negative eigenvalues, whereas the boundary of $H(e_2^n)$ contains points with as many as $n-2$ negative eigenvalues. Thus we get the following corollary.}

%\begin{corollary}
%If $n-1 > k > 2$ or $n = 4$ and $k = 2$, then the set of possible eigenvalue vectors for matrices in $\mathcal{S}^{n,k}$ is strictly contained in $H(e_k^n)$.
%\end{corollary}

%{\color{red}Should we need to modify the above corollary to the following --- based on the discussion in blue above?}
\begin{corollary}
If $n-1 > k \geq 2$, then the set of possible eigenvalue vectors for matrices in $\mathcal{S}^{n,k}$ is strictly contained in $H(e_k^n)$.
\end{corollary}

%On the other hand, we show a structure theorem for $2 < k < n-1$, which implies that all nonsingular matrices in $\mathcal{S}^{n,k}$, whose eigenvalues lie on the boundary of the hyperbolicity relaxation, will be essentially be equivalent to $G(n,k)$ up to a rescaling. In particular, there will be points on the boundary of the hyperbolicity relaxation that cannot be realized as the eigenvalues of a matrix in $\mathcal{S}^{n,n-1}$. 
%It is still open whether or not the corresponding set of possible eigenvalues is convex.
%{\color{blue} GB:There should be a precise theorem here.}
%
%
%Finally, we will describe a computational technique that can prove that the eigenvalues of matrices in $\mathcal{S}^{4,2}$ that lie on the boundary of the hyperbolicity relaxation are those that have at most 1 negative eigenvalue.
%{\color{blue} GB:There should be a precise theorem here.}
%\gb{Why was this a separate section and not subsection?}
\subsection{Eigenvalues of matrices in $\mathcal{S}^{4,2}$ whose eigenvalues lie on the boundary of $H(e^{4}_2)$}
Theorem \ref{thm:struc} implies that if $X$ is a nonsingular matrix in $\mathcal{S}^{n,k}$ whose eigenvalues lie on the boundary of $H(e^{n}_k)$, then $X$ is in fact diagonally congruent to the matrix $G(n,k)$. Because $G(n,k)$ has only one nonnegative eigenvalue, Sylvester's law of inertia \cite{Horn1985matrix} implies that a matrix of this form has at most one negative eigenvalue.

The next lemma is a converse to the previous observation when $n = 4$ and $k = 2$.
\begin{lemma}
If $\lambda \in H(e^4_2)$, so that $e^4_2(\lambda) = 0$ and $\lambda$ has at most one negative eigenvalue, then $\lambda$ is a vector of eigenvalues for matrix which is diagonally congruent to $G(4,2)$.
\end{lemma}

The idea of the proof is to use a characterization of the coefficients of real rooted polynomials to reduce the problem to proving that there exist real rooted polynomials having certain properties. These properties can be described entirely in terms of polynomial inequalities, so we can use algorithms for quantifier elimination over real closed fields to solve this problem. We do not know of a proof of this result that does not rely on computational methods.

\section{Conclusions and open questions}\label{sec:conclusion}
The key insight in this paper is the observation that $H(e^n_k)$, i.e., the hyperbolicity cone of the elementary symmetric polynomial \change{$e^n_k$,} is a convex relaxation of the set of the eigenvalue of matrices in $\mathcal{S}^{n,k}$. Using this insight, we are able to improve upper bounds on the distance of the matrices in $\mathcal{S}^{n,k}$ from PSD matrices $\mathcal{S}^n$ given in~\cite{blekherman2020sparse}. The next question that was considered is how good is the relaxation $H(e^n_k)$: We first show that this relaxation (apart from the trivial case of $n =k$) is tight for the case of $k = n -1$. Indeed, in this case, we are able to show that $H(e^n_{n -1})$ is exactly the set of eigenvalues of matrices in $\mathcal{S}^{n, n -1}$. However, in general we prove that if $\lambda(M)$ belongs to the boundary of $H(e^n_k)$ and $M \in \mathcal{S}^{n,k}$, then either $M$ is non-singular or $M$ is diagonally congruent to $G(n,k)$. Since there are points on the boundary of $H(e_k^n)$ with as many as $n - k$ negative entries and no zero entries, this shows that ``large parts" of the boundary of $H(e_k^n)$ \change{do} not intersect with the eigenvalues of matrices in $\mathcal{S}^{n,k}$.

There are many interesting open questions. As discussed above, we have shown that the set of eigenvalues of matrices in $\mathcal{S}^{n, n-1}$ is convex. It was recently shown in \cite{Khaz} that the set of eigenvalue vectors of matrices in $\mathcal{S}^{4,2}$ is not convex, but it is still an open question for all other values of $k < n -1$. Another question vis-\'a-vis the structure theorem is to classify singular matrices in $\mathcal{S}^{n,k}$ that lie on the boundary of $H(e^n_k)$.  Finally, instead of enforcing PSD-ness on all submatrices, we can enforce PSD-ness of a smaller set of submatrices. Are there similar relaxations like $H(e^n_k)$ for specially structured \change{collections} of submatrices?
    
\section{Preliminaries}\label{sec:pre}
\subsection{Introduction to \change{hyperbolic polynomials and hyperbolicity cones}}
Hyperbolic multivariate polynomials are a rich collection of polynomials with connections to convex optimization, combinatorics and theoretical computer science.  We give a brief description of the important properties of hyperbolic polynomials here. We say that a polynomial $p \in \R[x_1 ,\dots, x_n]$ is \emph{hyperbolic} with respect to a fixed vector $v$ if $p(v)>0$, and for any fixed $x \in \R^{n}$, the univariate polynomial $p(x - tv) \in \R[t]$ has only real roots. 
An important example of this comes from the determinant of symmetric matrices, 
$\det (X
%\delete{_{ij}}
) \in \R[x_{ij} : i \le j]$, which is hyperbolic with respect to the identity matrix, by the spectral theorem. Another example which is critical for our purposes is the elementary symmetric polynomial $e_k^n(x)$, defined as 
\[
    e_k^n(x) = \sum_{S \subseteq [n] : |S| = k} \prod_{i \in S} x_i.
\]
This polynomial is hyperbolic with respect to $\vec{1}$, the all ones vector.  %\gb{What is missing from the discussion below is that hyperbolicity cone consists of (the closure of) vectors with respect to which the polynomial is hyperbolic. Do we need to single out all ones vector?}.

Let $V(p)\subset \mathbb{R}^n$ be the set of zeros of the polynomial $p$. Let $p$ be hyperbolic with respect to $v\in \mathbb{R}^n$. The \change{closed} hyperbolicity cone of polynomial $p$ with respect to $v$, is the closure
%\footnote{\delete{some references define the hyperbolicity cone as just the connected component of $\mathbb{R}^n \setminus V(p)$ containing $v$}} 
of the connected component of $\R^n \setminus V(p)$ containing $v$~\cite{gaarding1951linear,guler1997hyperbolic,renegar2004hyperbolic}. We will denote it by $H_v(p)$. When $p$ is the determinant or sum of certain principal subdeterminants of a symmetric matrix, we will abbreviate $H(p)=H_I (p)$, where $I$ is the identity matrix. 

\change{The hyperbolicity cone of elementary symmetric polynomial $e_k^n(x)$ with respect to $\vec{1}$ will be denoted by $H(e_k^n)$. A simple algebraic characterization of $H(e_k^n)$ is given by (see for example~\cite{zinchenko2008hyperbolicity})
\[
    H(e_k^n)=\{x\in\R^n: e_l^n(x)\ge 0\textup{ for all }1\le l\le k\}. 
\]
It is also known that $H(e_k^n)$ is spectrahedral~\cite{Branden2014spec} for all $1\le k\le n$, i.e., an affine slice of a higher dimensional PSD cone. 
}

%where we suppress the dependence on $v$ for our purposes, as it will be clear from context. In the case of the determinant, the $H(\det)$ is the positive semidefinite cone.

The key fact for our purposes is Proposition 1 in \cite{renegar2004hyperbolic} originally proved by G\r{a}rding in \cite{gaarding1951linear}. 
\begin{lemma}\label{lem:hyperbolic_boundary}
    \change{If $p$ is hyperbolic with respect to $v$, then} $H_v(p)$ is convex, and its boundary is precisely $H_v(p) \cap V(p)$.
\end{lemma}

\subsection{Linear algebra}
We use the following standard results in the paper. \\Proofs can be found in \cite{Horn1985matrix}. 

\begin{theorem}[Schur-Horn theorem]\label{thm:Schur-Horn} Let  $d, \lambda \in \mathbb{R}^n$ such that $d_i \geq d_{i +1}$ and $\lambda_i \geq \lambda_{i +1}$, $i \in [n -1]$. There is a symmetric matrix with diagonal values $d$ and eigenvalues $\lambda$ if and only if:
\begin{itemize}
\item $\sum_{i = 1}^j d_i \leq \sum_{i = 1}^j \lambda_i$ for all $j \in [n -1]$,
\item $\sum_{i = 1}^nd_i = \sum_{i = 1}^n \lambda_i$.
\end{itemize}
\end{theorem}

\change{Note that the original theorem in~\cite{Horn1985matrix} proves the existence of a real symmetric matrix (not just Hermitian) that achieves the desired eigenvalues and diagonal entries. }

We next present the famous Cauchy's Interlacing Theorem. 
	
\begin{theorem}[Cauchy's Interlacing theorem]\label{thm:Cauchy} Consider an $n \times n$ symmetric matrix $A$ and let $A|_J$ be any of its $k\times k$ principal submatrices. Then for all $1\leq i\leq k$, 
$$ \lambda_{n-k+i}(A)\leq \lambda_i (A|_J)\leq \lambda_i (A). 	$$
\end{theorem}
% and {\displaystyle \mathbf {\lambda } =\{\lambda _{i}\}_{i=1}^{N}}\mathbf{\lambda}=\{\lambda_i\}_{i=1}^N be vectors in {\displaystyle \mathbb {R} ^{N}}\mathbb {R} ^{N} such that their entries are in non-increasing order. There is a Hermitian matrix with diagonal values {\displaystyle \{d_{i}\}_{i=1}^{N}}\{d_i\}_{i=1}^N and eigenvalues {\displaystyle \{\lambda _{i}\}_{i=1}^{N}}\{\lambda_i\}_{i=1}^N if and only if
%
%{\displaystyle \sum _{i=1}^{n}d_{i}\leq \sum _{i=1}^{n}\lambda _{i}\qquad n=1,2,\ldots ,N}\sum_{i=1}^n d_i \leq \sum_{i=1}^n \lambda_i \qquad n=1,2,\ldots,N
%and
%
%{\displaystyle \sum _{i=1}^{N}d_{i}=\sum _{i=1}^{N}\lambda _{i}.}\sum_{i=1}^N d_i= \sum_{i=1}^N \lambda_i.

%\ss{I mentioned that ours is only symmetric case. Not sure if this is want we want to do. }
We also present the symmetric case of Jacobi Complementary Minors formula. The proof for general case can be found in~\cite{MR1411115}:
\begin{theorem}[Jacobi's Complementary Minors Formula, symmetric case]\label{thm:compMinors} \label{thm:jacobi}
  Let $M$ be an invertible $n \times n$ matrix and $\emptyset \subsetneq S \subsetneq [n]$.
\begin{equation}\label{eqn:minordet}
    \det(M|_S) = \det(M)\det(M^{-1}|_{\change{S^c}}).
\end{equation}
\end{theorem}
This theorem states that the minor of $M$ corresponding to a subset $S$, can be written in terms of the minor of $M^{-1}$ with respect to the complement \change{$S^{c}$}.
In the simplest case when $|S| = n-1$, this is simply Cramer's rule for the diagonal entries of the inverse matrix. 

\section{Proof of Theorem~\ref{thm:min_eigen} and Corollary~\ref{cor:improve}}
%The theorem we wish to prove in this section is
%\begin{theorem}
%Let $F \in \mathcal{F}$ and let $\tilde{G}(n,k)=\frac{G(n,k)}{F(G(n,k))}$. For any $M\in \mathcal{S}^{n,k}$ with $F(M)=1$, the minimal eigenvalue of $M$ is at least as large as the minimal eigenvalue of $\tilde{G}(n,k)$:
%   $$\lambda_1(M)\leq \lambda_1(\tilde{G}(n,k)).$$
%\end{theorem}
\subsection{Proof of Theorem~\ref{thm:min_eigen}}
For the remainder of this section, fix $n > k \geq 2$ and an $F \in \mathcal{F}$. Let $f$ be the function so that $f(\lambda(M)) = F(M)$ for each $M \in \Sym_n$.

\change{In order to prove the theorem, we would like to verify an appropriate lower bound on $z^*$ defined as:
\begin{align*}
 z^{*}:=   \text{minimize }&\lambda_1(M)\\
    \text{s.t. }&f(\lambda(M)) = 1\\
                & M \in \mathcal{S}^{n,k},
\end{align*}
where $\lambda_1(M)$ is the smallest eigenvalue of $M$. 
%Note that if $z^* \geq 0$, then there is nothing to prove since $\lambda_1(\tilde{G}(n,k)) < 0$. Therefore we may assume $z^* < 0$. 

In order to provide a lower bound on $z^*$, we apply (i) the hyperbolic relaxation for the eigenvalues of $\mathcal{S}^{n,k}$ to replace  $\{\lambda(M)\,|\, M \in \mathcal{S}^{n,k}\}$ with $H(e^n_k)$, (ii) replace $f(\lambda(M))  = 1$ by $f(\lambda(M))  \leq 1$ to obtain the following convex optimization problem:

\begin{align}\label{eq:convrel}
\begin{array}{rl}
z^{l}:=    \text{minimize }&\lambda_1\\
    \text{s.t. }& \lambda_1 \leq \lambda_i \ \forall i \in \{2, \dots, n\}\\
		&f(\lambda) \leq 1\\
                &(\lambda_1, \lambda_2 ,\dots, \lambda_n) \in H(e_k^n).
\end{array}
\end{align}

It is straightforward to verify that the set $\{\lambda\,|\, f(\lambda) \leq 1, \lambda \in H(e^n_k)\}$ is compact. Thus $z^l$ is finite and at least one optimal solution exists. Also note that since (\ref{eq:convrel}) is a convex program which is symmetric with regards to variables $\lambda_2, \dots, \lambda_n$, it is straightforward to verify that there exists an optimal solution where $\lambda_2 = \dots = \lambda_n$. Therefore, we arrive at the following two variable optimization problem:
\begin{align}\label{eq:convrel1}
\begin{array}{rl}
z^{l}:=    \text{minimize }&\lambda_1\\
    \text{s.t. }& \lambda_1 \leq \lambda_2 \\
		&f(\lambda_1, \lambda_2, \dots, \lambda_2) \leq 1\\
                &(\lambda_1, \lambda_2 ,\dots, \lambda_2) \in H(e_k^n).
\end{array}
\end{align}
Next observe that if we remove the constraint $(\lambda_1, \lambda_2 ,\dots, \lambda_2) \in H(e_k^n)$ from (\ref{eq:convrel1}), then 
\begin{itemize}
    \item If $f$ corresponds to a norm, then the optimal solution of the resulting problem is of the form $(a, 0, \dots, 0)$ where $a <0$, which is infeasible for (\ref{eq:convrel1}) since it does not satisfy the constraint $(\lambda_1, \lambda_2 ,\dots, \lambda_2) \in H(e_k^n)$. Thus, the optimal solution of (\ref{eq:convrel1}) belongs to the boundary of of $H(e_k^n)$. %the constraints $(\lambda_1, \lambda_2 ,\dots, \lambda_2) \in H(e_k^n)$. 
    \item If $f$ corresponds to the trace function, then the optimal solution of the resulting problem is unbounded. Thus, again we can conclude that the optimal solution of (\ref{eq:convrel1}) belongs to the boundary of $H(e_k^n)$. %the constraints $(\lambda_1, \lambda_2 ,\dots, \lambda_2) \in H(e_k^n)$.
\end{itemize}
Simple computation shows that for $(\lambda_1, \lambda_2, \dots, \lambda_2)$ to be on the boundary of $H(e^n_k)$, we have $\lambda_1  = -\frac{n - k}{k}\lambda_2$. Thus, we obtain that 
$z^{l}  = -\frac{1}{f\left(-1, \frac{k}{n - k}, \dots, \frac{k}{n - k} \right) }$. Since $f$ is $1$-homogeneous, it is easy to verify that $\lambda_1(\tilde{G}(n,k)) = -\frac{1}{f\left(-1, \frac{k}{n - k}, \dots, \frac{k}{n - k} \right) }$, which completes the proof of the theorem}.

\subsection{Proof of Corollary~\ref{cor:improve}}
%\subsection{Improving theorem in old paper}
%Now we use this technique to improve a theorem in our older paper. 
%
%\begin{lemma}
%    Let $\lambda=(\lambda_1,\lambda_2,...,\lambda_n)\in H(e_k^n)$ such that $\|\lambda\|_2=1$. If $\min_i \lambda_i <0$ then $$|\min_i \lambda_i|\leq \frac{n-k}{\sqrt{(n-k)^2+(n-1)k^2}}.$$
%\end{lemma}
%
%\begin{proof}
%    Without loss of generality we may assume $\lambda_1$ is the minimum $\lambda_i$. We consider $\lambda'$ obtained from $\lambda$ by symmetrizing over all coordinates except for first one, i.e. $\lambda'_1=\lambda_1$ and $\lambda'_i=(\lambda_2+...+\lambda_n)/(n-1)$ for all $i\geq 2$. 
%    
%    From triangle inequality we get $\|\lambda'\|_2\leq \|\lambda\|_2$. Thus $\lambda'/\|\lambda'\|_2$ has unit $\ell_2$ norm and its minimum entry is smaller than $\lambda$. This shows that the optimal $\lambda$ has only one negative entry, and all other entries are the same. 
%    
%    Thus we can see that the optimal solution is a scaling of $(k-n,k,...,k)$, and we get our result. 
%\end{proof}
%
%\begin{remark}
%    This proof works for any symmetric norm such that norm of any standard basis vector is 1. In particular it holds for all $\ell_p$ norms. 
%\end{remark}

%\begin{corollary}
%    Let $A\in\mathcal{S}^{n,k}$ with $\|A\|_F=1$. Then $\textup{dist}(A,\mathcal{S}^n_+)\leq \frac{(n-k)^{3/2}}{\sqrt{(n-k)^2+(n-1)k^2}}$
%\end{corollary}
%\change{Deleted the proof environment}

   Let $A\in\mathcal{S}^{n,k}$ with $\|A\|_F=1$. If $A$ is PSD then the distance is zero, so we assume $A$ has at least one negative eigenvalue. \change{By the} Cauchy interlacing theorem (Theorem~\ref{thm:Cauchy}), $A$ has at most $n-k$ negative eigenvalues. So $\textup{dist}(A,\mathcal{S}^n_+)\leq |\lambda_{1}(A)|\sqrt{n-k}$.  By Theorem~\ref{thm:min_eigen} we have that $|\lambda_{1}(A)| \leq \frac{1}{ \sqrt{1 + (n-1) \frac{k^2}{(n -k)^2} }}$, which completes the proof.
%Since the eigenvalue belongs to hyperbolicity cone $H(e_k^n)$ we get the result using previous lemma. 

\section{Proof of Theorem~\ref{thm:k=n-1}}

\change{When $k=n$ the statement is clear, since $x\in H(e_n^n)$ if and only if $x\ge 0$, and it is the eigenvalues of $\diag(x)$ which is PSD. 
When $k=1$, let $x\in H(e_1^n)=\{y:\sum_{i=1}^n y_i\ge 0\}$. By the Schur-Horn theorem (Theorem~\ref{thm:Schur-Horn}), there exists a symmetric matrix $M_0$ with identically zero diagonal entries and eigenvalues $x-\frac{\sum_{i=1}^n x_i}{n}$. Thus, $M_0+\frac{\sum_{i=1}^n x_i}{n}I$ has eigenvalues $x$ and is in $\mathcal{S}^{n,1}$, since all of its diagonal entries are nonnegative. 
}

%We show that the hyperbolicity cone relaxation of $\mathcal{S}^{n,n-1}$ is exact on the level of eigenvalues. 

%That is, \gb{Why do we have a different name for Theorem 2.3 (and now we call it a lemma)? We should either say that we restate Theorem 2.3 or not do a precise statement.}

%\begin{lemma}
    %If $x \in H(e_{n-1}^n)$, then $x$ is the set of eigenvalues of some matrix in $\mathcal{S}^{n,n-1}$.
%\end{lemma}
\change{Now let $k=n-1$. }First, we reduce to the case when $x$ is on the boundary of $H(e_{n-1}^n)$. To do that, we note that if $x$ is any point in $H(e_{n-1}^n)$, then for some $t>0$, $x - t \vec{1}$ will lie on the boundary of the cone. If $x - t\vec{1} $ is \change{a vector of eigenvalues} of a matrix $M$ in $\mathcal{S}^{n,n-1}$, then $x$ is \change{a vector of eigenvalues} of $M + t I$, also in $\mathcal{S}^{n,n-1}$.

We apply a corollary 3.2 in \cite{zinchenko2008hyperbolicity} to argue that if $x$ is in $H(e_{n-1}^n)$, then it has at most one negative eigenvalue. 
\begin{lemma}
    If $x \in H(e_{n-1}^n)$, and $x$ has a negative entry, then all other entries of $x$ are strictly positive.
\end{lemma}
\begin{proof}
\change{
    By Theorem 1.1 in \cite{sanyal2011derivative}, we have that $x \in H(e_{n-1}^n)$ if and only if
    \[
        X = \diag(x_1, \dots, x_{n-1}) + x_n \vec{1}_{n-1}\vec{1}_{n-1}^{\intercal} \succeq 0.
    \]
    Here, we use $\vec{1}_{n-1}$ to denote the all ones vector in $n-1$ dimensions for emphasis.
    
    By permuting the coordinates, we can assume that $x_n$ is an entry so that $x_n < 0$. Then, we have that the diagonal entries of $X$ are nonnegative, so for $i \neq n$,
    \[
        x_i + x_n \ge 0.
    \]
    So, $x_i \ge -x_n > 0$, concluding the theorem.
}

\end{proof} 

Thus, if $x$ lies on the boundary, we will consider two cases: either all entries of $x$ are nonnegative, or exactly one entry of $x$ is negative and others are positive. If all entries of $x$ are nonnegative, then there is a PSD matrix whose vector of eigenvalues is $x$, and in particular, a matrix in $\mathcal{S}^{n,n-1}$ with these eigenvalues.

If $x$ lies on the boundary and exactly one entry of $x$ is negative, then consider
\[
    e_{n-1}^n(x) = \sum_{i \in [n]} \prod_{j \in [n] \setminus i} x_j = \left(\prod_{j \in [n]} x_j\right)\sum_{i \in [n]}\frac{1}{x_i} = 0,
\]
which is well defined since in the previous lemma, we showed that all entries of $x$ are nonzero. Thus, $\sum_{i \in [n]}\frac{1}{x_i} = 0$.

Now, we can apply the Schur-Horn theorem (Theorem~\ref{thm:Schur-Horn}), which implies that there is a matrix $L$ whose diagonal \change{entries are} all zeros, \change{and }whose eigenvalues are $\{\frac{1}{x_i}\}$. In particular, $L$ is invertible, so let $M = L^{-1}$. Since all of the diagonal entries of $L$ are 0, all of the $(n-1) \times (n-1)$ minors of $M$ are zero by Cramer's rule for the diagonal entries of the inverse matrix. Also note that $x$ has $(n-1)$ positive entries, so by eigenvalue interlacing, all of the $(n-1)\times(n-1)$ minors of $M$ have at most 1 non-positive eigenvalue. Now, simply by noting that they all have 0 as an eigenvalue, this in particular implies that all $(n-1)\times(n-1)$ minors of $M$ have nonnegative eigenvalues, and hence are PSD. Thus, $M$ is a matrix in $\mathcal{S}^{n,n-1}$ with the desired eigenvalues.

\section{Proof of the structure theorem for $\mathcal{S}^{n, k}$}
\subsection{Proof roadmap}
Given a matrix in $\mathcal{S}^{n,k}$, which is nonsingular and locally singular  -- we will 
%\delete{we will} 
abbreviate by saying that $M$ is an \textbf{NLS} matrix.

We show Theorem \ref{thm:struc} in three steps. We first prove base cases $n-k=2$ and $k=3$, and then we use double induction on $n-k$ and $k$ to prove the statement for general $k$.
%We will show this lemma using induction on two variables: firstly on $n-k$, then on $n$. We will use some facts about Schur complements to complete the inductive step. 
For the base case $n-k = 2$, there is a very interesting step of taking the inverse of a given \change{NLS} matrix, and using some facts about the structure of the inverse matrix.

The inductive step for this argument relies on some observations about Schur complements. The Schur complement of a symmetric matrix $M$ with respect to a nonzero diagonal entry $M_{ii}$ is defined to be {\color{black} the $(n-1)\times (n-1)$ matrix}
\[
    M \setminus \{i\} = M|_{[n] \setminus \{i\}} - \frac{1}{M_{ii}} \tilde{M}_{i} \tilde{M}_{i}^{\top},
\]
where {\color{black} $\tilde{M}_i$ is obtained from the $i^{th}$ column of $M$ after removing $i^{th}$ entry.}

Now, we recall some facts for matrices $M$ with strictly positive diagonal entries~\cite{zhang2006schur}.
%TODO: Cite something here.
\begin{itemize}
    \item \textbf{Schur complements preserve PSD-ness.} $M$ is PSD if and only if $M \setminus \{i\}$ is PSD.
    \item \textbf{Schur complements preserve singularness.} $M$ is nonsingular if and only if $M \setminus \{i\}$ is nonsingular.
    \item \textbf{Schur complements commute with taking submatrices.} If $i \notin S$, then {\color{black}$(M \setminus \{i\})|_S = (M|_{S\cup \{i\}})\setminus \{i\}$.}
\end{itemize}
The previous three properties imply the following: a matrix in $\mathcal{S}^{n,k}$ is NLS, if and only if for each $i\in [n]$, $M\setminus \{i\}$ is in $\mathcal{S}^{n-1,k-1}$ and NLS.

%Now, we see that if we know all NLS matrices in $\mathcal{S}^{n,k}$ are diagonally congruent to $G(n,k)$, then this puts enormous constraints on NLS matrices in $\mathcal{S}^{n+1, k+1}$, as their Schur complements will need to be congruent to $G(n,k)$. This will be the structure we exploit to prove the stronger structure theorem. \ss{Not sure if we put this paragraph here. Maybe instead just mention this Schur complement is our basic idea for induction}

\subsection{Structure theorem when $k = n-2$}
Let $M$ be a matrix in $\mathcal{S}^{n,n-2}$, which is NLS. Observe that NLS matrices in $\mathcal{S}^{n,n-2}$ must have strictly positive diagonal entries. If any diagonal entry is zero, then since $2\times 2$ minors of $M$ are nonnegative, an entire row and column of $M$ are filled with zeros, and then $M$ is singular, which is a contradiction. %So, we can assume that all diagonal entries of $M$ are strictly positive.

As a base case when $n=4$, consider an NLS matrix $M \in \mathcal{S}^{4,2}$.  We can perform a diagonal congruence transformation to obtain a matrix $\tilde{M}$, such that all of the diagonal entries of $\tilde{M}$ are 1, and since all $2\times 2$ minors of $M$ are zero, we see that all off diagonal entries of $\tilde{M}$ are $\pm 1$.  There are 6 off-diagonal entries, so that there are 64 distinct possibilities for locally singular matrices, up to \change{diagonal} congruence. All of these 64 matrices are either singular or congruent to $G(4,2)$, which can be checked using direct computation.

\begin{lemma}\label{lem:3cond}
Let $M \in \mathcal{S}^{n,n-2}$ be an NLS matrix. Then the following hold:
\begin{enumerate}
    \item $\det(M) < 0$.
    \item All $(n-1)\times(n-1)$ principal minors of $M$ are strictly negative.
    \item All $(n-3)\times(n-3)$ principal minors of $M$ are strictly positive.
\end{enumerate}
\end{lemma}
\begin{proof}
We prove these facts by inducting on $n$, with the base case $\mathcal{S}^{4,2}$ following from direct checking of the $64$ cases above. For the inductive step we take the Schur complement of an NLS matrix $M$ with respect to a diagonal entry. Observe that a diagonal entry of the Schur complement cannot be zero.
%During this process, as long as the size of the matrix is at least 4, the diagonal entries are never zero, since 
Otherwise the whole row of the Schur complement must be zero as all $2\times 2$ minors are nonnegative, and this is a contradiction since $M$ is nonsingular, \change{ and Schur complements preserve nonsingularity}. %Thus the above statements follow by induction. 
Since taking Schur complements with respect to a positive diagonal entry preserves the property of being NLS, preserves the signs of determinants, and commutes with the operation of taking submatrices, all three above statements follow by induction.
\end{proof}
%if these 3 statements hold for $\mathcal{S}^{n,n-2}$ for some $n$, then it holds for all larger $n$. Thus, it suffices to check this for the $\mathcal{S}^{4,2}$ case, where it is clear from checking the finite number of cases.
Now we are ready to prove the main theorem of this Section.

\begin{theorem}\label{thm:kn-2}
Let $M\in \mathcal{S}^{n,n-2}$ be an NLS matrix. Then $M$ is diagonally congruent to $G(n,n-2)$.
\end{theorem}

\begin{proof}
Let $M\in \mathcal{S}^{n,n-2}$ be an NLS matrix and consider the inverse matrix $M^{-1}$. Using Lemma~\ref{lem:3cond} and Theorem~\ref{thm:jacobi} and what we know about principal minors of $M$, we have the following: 

\begin{enumerate}
    \item All diagonal entries of $M^{-1}$ are strictly positive. 
    
    \item All $2\times 2$ principal minors of $M^{-1}$ are zero. 
    
    \item All $3\times 3$ principal minors of $M^{-1}$ are strictly negative. 
\end{enumerate}
%The diagonal entries of $\mathcal{S}^{n,2}$ are positive, since they have the form $\frac{\det(M_{-i})}{\det(M)}$, and this is the ratio of two quantities with the same sign by Cramer's rule. By lemma~\ref{lem:minordet}, all the $2\times 2$ minors of $M^{-1}$ are 0. This implies that $M^{-1}$ is diagonally congruent to a $\pm 1$ matrix in $\mathcal{S}^{n,2}$. 

Observe that (1) and (2) together imply that all off-diagonal entries of $M^{-1}$ are nonzero. We conjugate $M^{-1}$ by a diagonal matrix $D$ given by $D_{11}=-1/\sqrt{(M^{-1})_{11}}$ and $D_{ii}=\operatorname{sgn}(M_{1i})
/\sqrt{(M^{-1})_{ii}}.$ to obtain matrix $T$ with $1$'s on the diagonal, and $-1$'s in first row and column other than the $(1,1)$ entry. By \change{Theorem} \ref{thm:jacobi}, all $2\times 2$ principal minors of $T$ are zero, so of its off-diagonal entries must be $\pm 1$.

%Now, suppose that the entry $(i,j)$ in $-M^{-1}$ were 1 for some $i, j \neq 1$, then we could consider the submatrix indexed by $\{1,i,j\}$, which can easily be seen to be all $1$'s. This is a singular $3\times 3$ matrix in $M^{-1}$, which implies the existence of a singular $(n-3)\times(n-3)$ matrix in $M$. This is a contradiction of the third statement above. Thus, we see that all of the $(i,j)$ entries for $i,j \neq 1$ are -1, and so $M^{-1}$ is diagonally congruent to $G(n,2)$.

Now for all distinct $i,j\neq 1$, we consider the principal submatrix with rows and columns indexed by $\{1,i,j\}$. It has form $\begin{pmatrix}1 & -1 & -1\\-1 & 1 & x \\-1 & x & 1\end{pmatrix}$ where $x$ is either $1$ or $-1$. Since this submatrix has negative determinant, we must have $x=-1$. Thus $T$ is $G(n,2)$, and $M^{-1}$ is diagonally congruent to $G(n,2)$. 

Finally, notice that inverting a matrix sends diagonally congruent matrices to diagonally congruent matrices, and that for $n \ge 4$,
\[
    \change{G(n,2)^{-1} = \frac{n-3}{2(n-2)}G(n,n-2).}
\]

Thus, we have shown that $M$ is diagonally congruent to $G(n,n-2)$, as desired.

\end{proof}
\subsection{Structure theorem for $k = 3$} In this section we prove the structure theorem for $k=3$:
\begin{theorem}\label{thm:k3}
Let $M\in \mathcal{S}^{n,3}$ with $n \geq 5$ be an NLS matrix. Then $M$ is diagonally congruent to $G(n,3)$.
\end{theorem}

%Another extremal case is $k = 3$. 
Our proof proceeds by induction on $n$. As a base case, note that the result holds for $\mathcal{S}^{5,3}$ by Theorem \ref{thm:kn-2}.

To finish the induction we need the following lemma. 

\begin{lemma}\label{lem:n-1orn-2}
    If $M$ is a nonsingular symmetric matrix, then either one of its $(n-1)\times(n-1)$ principal minors is nonzero, or one of its $(n-2)\times(n-2)$ \change{principal} minors is nonzero.
\end{lemma}
\begin{proof}
    Let $M$ be a nonsingular matrix and consider $M^{-1}$. If all of the principal $(n-1)\times(n-1)$ minors of $M$ are zero, then by Theorem~\ref{thm:jacobi} all of the diagonal entries of $M^{-1}$ are 0. If, in addition, all $(n-2) \times (n-2)$ principal minors of $M$ are zero, then all $2\times 2$ minors of $M^{-1}$ are zero, and then $M^{-1}$ is the zero matrix, which is a contradiction.
\end{proof}

Now  \change{for an inductive hypothesis, assume that for $5 \le m < n$, any NLS $M \in S^{m,3}$ is diagonally congruent to $G(m,3)$. } Fix $n$ and let $M\in \mathcal{S}^{n,3}$ be any NLS matrix. If the submatrix \change{ of $M$} given by Lemma \ref{lem:n-1orn-2} has size at least 5, then it must be diagonally congruent to $G(n-1,3)$ or $G(n-2,3)$ due to the inductive hypothesis. We divide the remaining proof into three cases: 

\begin{enumerate}
    \item $n\geq 6$, and there exists an $(n-1)\times(n-1)$ submatrix of $M$ \change{that} is diagonally congruent to $G(n-1,3)$. Then after permutation and suitable diagonal congruence we may assume
\[
    M = 
    \left(
    \begin{array}{c|c}
        G(n-1,3) & v\\
        \hline
        v^{\top} & 1
    \end{array}
\right),
\]
    for some vector $v\in \R^{n-1}$.

    Let $M'$ be any $5\times 5$ principal submatrix of $M$ which includes index $n$. \change{Then} $M'$ must have the form
    \[
    M' = 
    \begin{pmatrix}
        1 & -\frac{1}{2} & -\frac{1}{2} & -\frac{1}{2} & v_1\\
        -\frac{1}{2} & 1 & -\frac{1}{2} & -\frac{1}{2} & v_2\\
        -\frac{1}{2} & -\frac{1}{2} & 1 & -\frac{1}{2} & v_3\\
        -\frac{1}{2} & -\frac{1}{2} & -\frac{1}{2} & 1 & v_4\\
        v_1 & v_2 & v_3 & v_4 & 1
    \end{pmatrix}.
\]
If we look at the $3\times 3$ submatrix corresponding to entries $\{i, j, 5\}$, we get
\[
    \begin{pmatrix}
        1 & -\frac{1}{2} & v_i\\
        -\frac{1}{2} & 1 & v_j\\
        v_i & v_j & 1
    \end{pmatrix}.
\]

The determinant of this matrix is 
\begin{equation}\label{eq:S53}
    \frac{3}{4}-v_i^2-v_iv_j-v_j^2.
\end{equation}

Because all $3\times 3$ submatrices of $M'$ are singular, this determinant must equal 0 for all $i, j \in \{1,2,3,4\}$.
Notice that this is a quadratic equation in $v_i$ and $v_j$. If we fix a value for $v_1$, then (\ref{eq:S53}) implies that the remaining three $v_i$ can take on at most 2 other values (which only \change{depend} on $v_1$). By \change{the} pigeonhole principle, at least two of these $v_i$ must be equal. After permuting entries we may assume $v_2=v_3$. Plugging this into (\ref{eq:S53}), we see that either $v_2=v_3= \frac{1}{2}$ or $v_2=v_3= -\frac{1}{2}$. In the first case we may conjugate $M'$ by $\diag(1,1,1,1,-1)$. Therefore we may assume $v_2=v_3=-\frac{1}{2}$. 

Now, we can consider the equation (\ref{eq:S53}) for the cases when $i = 2$ and $j = 1$, or $i = 2$ and $j = 4$.  Because we assume $v_2 = -\frac{1}{2}$, equation \ref{eq:S53} implies that 
\[
    \frac{3}{4}-v_2^2-v_1v_2-v_1^2 = 
    -(v_1 - 1)(v_1 + \frac{1}{2}),
\]
and
\[
    \frac{3}{4}-v_2^2-v_2v_4-v_4^2 = -(v_4 - 1)(v_4 + \frac{1}{2}).
\]

We then get that both $v_1,v_4$ are either $1$ or $-\frac{1}{2}$. They cannot both be 1, otherwise the equation fails for $i=1,j=4$. Therefore at least one of them must be $-\frac{1}{2}$, and after permuting entries we may assume $v_1=-\frac{1}{2}$. 

Summarize above, we see that $M'$ can only take on two values \change{up to diagonal congruence and permutation}: either $M' = G(5,3)$, or

\[
    M' = 
    \begin{pmatrix}
        1 & -\frac{1}{2} & -\frac{1}{2} & -\frac{1}{2} & -\frac{1}{2}\\
        -\frac{1}{2} & 1 & -\frac{1}{2} & -\frac{1}{2} & -\frac{1}{2}\\
        -\frac{1}{2} & -\frac{1}{2} & 1 & -\frac{1}{2} & -\frac{1}{2}\\
        -\frac{1}{2} & -\frac{1}{2} & -\frac{1}{2} & 1 & 1\\
        -\frac{1}{2} & -\frac{1}{2} & -\frac{1}{2} & 1 & 1
    \end{pmatrix}.
\]
\change{Now, because this holds for all $5\times 5$ submatrices of $M$, it is clear that $v$ must have the properties that all entries of $v$ are either $-\frac{1}{2}$ or $1$, and that at most one entry of $v$ can be $1$. If the $i^{th}$ entry of $v$ is 1, then notice that rows $i$ and row $n$ of $M$ are the same, meaning that $M$ is singular, a contradiction. We conclude that all entries of $v$ are $-\frac{1}{2}$, and we have shown that $M$ is diagonally congruent to $G(n,3)$.}

%%%%%%%%%%%%%%%%%%%%%%%%%%%%%%%%%%%%%%%%%%%%%%%%%%%%%%%%%%%%%%%%%%%%%%%%%%%%%%%%%%%%%%%%%%%%%%%%%%%%%%%%%%%%%%%%%%%%%%%%%%%%%%%%%%%%%%%%%%%%%%%%%%%

    \item $n\geq 7$, and there exists an $(n-2)\times(n-2)$ submatrix of $M$ which is nonsingular. By induction, this implies that this submatrix is diagonally congruent to $G(n-2,3)$. Then after permutation and suitable diagonal congruence we may assume
    \[
    M = 
    \begin{pmatrix}
        G(n-2,3) & v & w\\
        v^{\top} & 1 & x\\
        w^{\top} & x & 1\\
    \end{pmatrix}.
\]

If either $v$ or $w$ has all entries $-\frac{1}{2}$, then $M$ has an $(n-1)\times (n-1)$ principal submatrix equal to $G\big((n-1),3\big)$, and we are back to the previous case. 

Upon considering any $5\times 5$ principal submatrix of $M$ that has exactly one index from $\{n-1,n\}$, and using observations from the previous case, we may assume $v$ and $w$ to both have exactly one entry that is 1, with the remaining entries are $-\frac{1}{2}$, and $x$ is some scalar number. There are two cases of interest: either $v$ and $w$ have the 1 entry in the same position, or in different positions.

If they are both in the same place, then without loss of generality, let us assume that they are in position $(n-2)$. Now, if we look at the $3\times 3$ block corresponding to entries $\{n-2, n-1, n\}$, then we will see the $3\times 3$ matrix
\[
    \begin{pmatrix}
        1 & 1 & 1\\
        1 & 1 & x\\
        1 & x & 1
    \end{pmatrix}.
\]
The determinant of this matrix is $-(x-1)^2$. We can see that if this matrix is singular, then $x$ must in fact be equal to $1$, and so we see that the last 3 rows of $M$ are all the same, implying $M$ is singular. 
Now, suppose that $v$ and $w$ have these 1 entries in two different positions. Then we see that there is a $3\times 3$ submatrix of the form
\[
    \begin{pmatrix}
        1 & 1 & -\frac{1}{2}\\
        1 & 1 & x\\
        -\frac{1}{2} & x & 1
    \end{pmatrix}.
\]

The determinant of this matrix is $1+x-x^2-\frac{1}{4}-1 = -(x+\frac{1}{2})^2$, and we must then have $x=-\frac{1}{2}$.
 In this case, we see that the $(n-2)$ and $(n-3)$ rows of $M$ are equal, and so $M$ is singular.
 In other words, if $M \in \mathcal{S}^{n,3}$ is locally singular, and $M$ is nonsingular, and some $(n-2)\times(n-2)$ minor of $M$ is diagonally congruent to $G(n-2,3)$, then $M$ is diagonally congruent to $G(n,3)$.
%\subsection{Proof of conjecture for matrices in $\mathcal{S}^{n,3}$.}
%We first require the fact, proven in a previous result that all nonsingular, locally singular matrices in $\mathcal{S}^{6,3}$ are congruent to $G(6,3)$.  We can now proceed by induction. For each $n > 6$, if $M$ is nonsingular, locally singular in $\mathcal{S}^{n,3}$, then it contains a nonsingular submatrix of size $(n-1)\times (n-1)$  or $(n-2)\times(n-2)$. By induction, we see that both of these types of matrices must be congruent to $G(n-1,3)$ and $G(n-2,3)$ respectively. This implies that $M$ is diagonally congruent to $G(n,3)$ by the previous discussion.

%%%%%%%%%%%%%%%%%%%%%%%%%%%%%%%%%%%%%%%%%%%%%%%%%%%%%%%%%%%%%%%%%%%%%%%%%%%%%%%%%%%%%%%%%%%%%%%%%%

\item $n=6$, and all $(n-1)\times (n-1)$ principal minors of $M$ are zero. Then using Theorem~\ref{thm:jacobi}, all diagonal entries of $M^{-1}$ are zero. Since $M\in\mathcal{S}^{6,3}$ is NLS, again using \change{Theorem~\ref{thm:jacobi}}, we also see that all $3\times 3$ minors of $M^{-1}$ are zero. Any $3\times 3$ \change{submatrix of $M^{-1}$} must be of the form
$$\begin{bmatrix}
0 & a & b\\
a & 0 & c\\
b & c & 0
\end{bmatrix},
$$

which has determinant $2abc$. Since the determinant must be 0, this means that there cannot be any $3\times 3$ submatrix of $M^{-1}$ where all off-diagonal entries are nonzero. 
 Now we define an edge coloring on $K_6$, the complete graph with 6 vertices, as follows. An edge $(i,j)$ is colored red if $(M^{-1})_{ij}=0$, and blue otherwise. Our previous result shows that there cannot be any blue triangles in this colored graph. Therefore using the fact that the Ramsey number $R(3,3)$ is at most $6$ \cite{graham1990ramsey}, there must exist a red triangle. 
 
 In other words, there must exist an identically zero $3\times 3$ submatrix within $M^{-1}$. After permuting rows and columns we may assume its index to be $\{1,2,3\}$. 
 Now consider the submatrix of $M^{-1}$ indexed by $\{1,2,3,4,5\}$. The span of first three rows is at most two dimensional, so this submatrix is singular. Using Theorem~\ref{thm:jacobi} we get $M_{66}=0$. But this is a contradiction since all diagonal entries of $M$ must be nonzero. 

\end{enumerate}

\subsection{Structure theorem in general}
%We have shown the structure theorem in the cases when $k = 3$ or $n - k = 2$. Now, if $k > 3$ and $n - k > 2$, then we can first induct on $k$ by letting $M \in \mathcal{S}^{n,k}$ be NLS. The Schur complement of $M$ with respect to any diagonal entry is NLS in $\mathcal{S}^{n-1,k-1}$, and all of the $(n-2)\times(n-2)$ minors of the Schur complement of $M$ are nonsingular. This implies that all of $(n-1)\times(n-1)$ minors of $M$ are nonsingular. Thus, if we consider any $(n-1)\times(n-1)$ submatrix of $M$, we see that it is in $\mathcal{S}^{n-1,k}$, and by our inductive hypothesis for $n-k-1$, all $(n-1)\times(n-1)$ submatrices of $M$ are congruent to $G(n-1,k)$. Since this is true for any $(n-1)\times(n-1)$ submatrix of $M$, we conclude that $M$ must also be congruent to $G(n,k)$. This concludes the proof.

We have shown the structure theorem in the cases when $k = 3$ or $n - k = 2$. Now we use induction to prove the general case. 

\begin{theorem}
Fix integers $n\geq 5$ and $3\leq k\leq n-2$. Let $M\in \mathcal{S}^{n,k}$ be NLS. Then $M$ is diagonally congruent to $G(n,k)$. 
\end{theorem}

\begin{proof}
    We will use induction. The base cases are when $k=3$ or $k=n-2$, and they are already proven. These include all cases when $n=5$ or $n=6$. 
    
    For induction, fix $n\geq 7$ and $3<k<n-2$. Assume the theorem statement holds for $(n-1,k)$ and $(n-1,k-1)$. Let $M\in\mathcal{S}^{n,k}$ be NLS. The Schur complement of $M$ with respect to any diagonal entry is NLS in $\mathcal{S}^{n-1,k-1}$, and is therefore diagonally congruent to $G(n-1,k-1)$. Because all $(n-2)\times (n-2)$ \change{principal} submatrices of $G(n-1, k-1)$ are \change{nonsingular} for $k < n-1$, all of the $(n-2)\times(n-2)$ minors of the Schur complement of $M$ are nonsingular. This implies that all of $(n-1)\times(n-1)$ minors of $M$ are nonsingular, since Schur \change{complements} preserve the property of being singular. 
    
    Thus, if we consider any $(n-1)\times(n-1)$ principal submatrix of $M$, we see that it is NLS in $\mathcal{S}^{n-1,k}$, and by our inductive hypothesis, all $(n-1)\times(n-1)$ submatrices of $M$ are diagonally congruent to $G(n-1,k)$. This in particular shows that all entries of $M$ must be nonzero, and all diagonal entries strictly positive. 
    
    %\ss{Please double check remaining of this section, since I changed things back after redefining $G(n,k)$ to have all one diagonals. }
    
    Let $D'$ be a non-singular diagonal matrix so that $DM|_{\{1,\dots, n-1\}}D = G(n-1,k)$. Since we may freely choose between $D'$ and $-D'$, without loss of generality we may assume $D'_{11}$ is negative. Let $c=\frac{\textup{sgn}(M_{1n})}{\sqrt{M_{nn}}}$ (where $\textup{sgn}(x)$ is $-1$ if $x$ is negative, 1 if $x$ is positive, and $\textup{sgn}(0)=0$). We then have
       $$\left [ \begin{array}{rl} D' & 0\\ 0 & c \end{array}\right]M\left [ \begin{array}{rl} D' & 0\\ 0 & c\end{array}\right] = M' = \left [\begin{array}{rl} G(n-1,k) & v\\ v^{\top} & 1 \end{array}\right],$$
    for some vector $v$, and we know $v_1<0$. Our goal now is to show that all entries of $v$ must be $-\frac{1}{k-1}$, and $M$ is therefore diagonally congruent to $G(n,k)$. 
    
    Consider any size $n-1$ principal submatrix of $M'$ containing columns $1$ and $n$. It is diagonally congruent to $G(n-1,k)$ so there exists diagonal matrix $D=\textup{diag}(d_1,...,d_n)$ such that
%    $$
%    D\begin{pmatrix}
%        G(n-2,k) & \hat{v}\\ \hat{v}^{\top} & 1
%    \end{pmatrix}}D=G(n-1,k)
%    $$
    $$
    D\left[\begin{array}{rl}
        G(n-2,k) & \hat{v}\\ \hat{v}^{\top} & 1
    \end{array}\right]D=G(n-1,k),
    $$
where $\hat{v}$ is obtained from $v$ by truncating one entry other than the first coordinate, and $\hat{v}_1<0$. \change{W}e may also choose $d_1>0$. Now comparing diagonal entries of both sides we get $d_i^2=1$ for all $i$. Now for all $i>1$ the $(1,i)$ entry on both sides is negative, so $d_id_1>0$ for all $i>1$. This shows in fact $D=I$, and all entries of $\hat{v}$ are $-\frac{1}{k-1}$. Now varying over all possible choices of principal submatrices containing columns $1$ and $n$, we see all entries of $v$ must be $-\frac{1}{k-1}$. This concludes the proof. 
\end{proof}

\section{Eigenvalues of Locally Singular Matrices in $\mathcal{S}^{4,2}$}\label{sec:eigen42}
In the previous section, we found that all locally singular matrices in $\mathcal{S}^{4,2}$ are either singular or congruent to $G(n,k)$. In this section, we consider the eigenvalues of NLS matrices in $\mathcal{S}^{4,2}$.

In general, we may ask the following question: what are the possible eigenvalues of a matrix of the form $DG(n,k)D$, where $D$ is a nonsingular diagonal matrix. We know that $DG(n,k)D$ is locally singular and in $\mathcal{S}^{n,k}$, which implies that its eigenvalues lie on the boundary of $H(e^{n}_k)$.

Furthermore, by Sylvester's law of inertia, for any nonsingular diagonal matrix $D$, $DG(n,k)D$ has exactly one negative eigenvalue, and the remainder are positive. Hence, if $\lambda$ is the eigenvalue vector of a $DG(n,k)D$, then $\lambda$ has exactly one negative entry. We 
%\delete{would} 
conjecture that this is in fact sufficient for $\lambda$ to be the vector of eigenvalues for an NLS matrix in $S^{n,k}$. 
\begin{conjecture}
    If $\lambda \in H(e^n_k)$, $e^n_k(\lambda) = 0$, and $\lambda$ has at most 1 negative entry, then $\lambda$ is a vector of eigenvalues for $DG(n,k)D$ for some diagonal matrix $D$.
\end{conjecture}

As evidence for this conjecture, we will give a computational proof of the following theorem:
\begin{theorem}
    If $\lambda \in H(e_2^4)$ lies on the boundary of the hyperbolicity cone of $e_2^4$ and $\lambda$ has exactly 1 negative entry, then $\lambda$ is an eigenvalue vector of some matrix in $\mathcal{S}^{4,2}$.\label{thm:eigenvalues_s42}
\end{theorem}
We prove this by converting the question into a question about real rooted polynomials. We should think of these as being the characteristic polynomials of certain types of symmetric matrices, and these characteristic polynomials completely characterize their eigenvalues. \change{We defer the proofs of these characterizations to the appendix.}
%\ks{Many of the proofs have been moved to an appendix.}

We say that a univariate polynomial  of degree $4$, $p = a_0 + a_1 x + a_2 x^2 + a_3 x^3 + x^4$, has \textbf{ good roots} if it is real rooted, $a_0 < 0$, $a_2 =0$ and $a_3 \le 0$.

\begin{lemma}\label{lem:good_roots}
\change{A real rooted polynomial} $p$ has good roots if and only if $p$ has no zero roots, exactly one negative root, and the roots of $p$ lie on the boundary of $H(e_2^4)$.
\end{lemma}

We then say that a polynomial $p = a_0 + a_1x + a_2 x^2 + a_3x^3 + x^4$ is \textbf{almost-nonnegative rooted} if there is some $k \in \R$ so that the polynomial $q = \frac{a_0}{-16} + \frac{a_1}{-4}x + kx^2+a_3x^3+x^4$ has nonnegative real roots.

\begin{lemma}\label{lem:almost_nonnegative_roots}
$p$ is almost-nonnegative rooted if and only if there is some nonsingular diagonal matrix $D$, so that $p$ is the characteristic polynomial of $DG(4,2)D$.
\end{lemma}

Now, Theorem \ref{thm:eigenvalues_s42} is easily seen to be equivalent to the following lemma.

\begin{lemma}\label{lem:good_implies_almost_nonnegative}
A polynomial $p$ has good roots if and only if it is almost-nonnegative rooted.
\end{lemma}

\change{We will prove Lemma \ref{lem:good_implies_almost_nonnegative} precisely in the appendix, but sketch the ideas here.
In principle, Lemma \ref{lem:good_implies_almost_nonnegative} is a statement in the first order theory of real closed fields.
That is, it can expressed entirely in terms of universal and existential quantifiers applied to real polynomial inequalities.
Such questions are well known to be answerable algorithmically through quantifier elimination techniques.
The first such algorithm for deciding such statements was found by Tarksi and Seidenberg, and further developments in this field can be found, for example in \cite{bochnak2013real}.
We used the quantifier elimination methods in Mathematica\cite{Mathematica} to solve this problem.

The main technical difficulty in applying these quantifier elimination methods is reducing the number of variables needed to express the inequalities so that the problem becomes tractable on a computer.
For this purpose, we prove a number of polynomial inequalities in the coefficients of a degree 4 univariate polynomial which imply both good-rootedness and almost-real-rootedness in the appendix.
Once these polynomial inequalities have been proven, the problem can be directly solved by a computer.}

\section*{Acknowledgments}
\change{We wish to thank the reviewers for their valuable comments which improved the paper and simplified some proofs.}

\bibliographystyle{plain}
\bibliography{refs}

\begin{thebibliography}{10}

\bibitem{ahmadi2019dsos}
Amir~Ali Ahmadi and Anirudha Majumdar.
\newblock Dsos and sdsos optimization: more tractable alternatives to sum of
  squares and semidefinite optimization.
\newblock {\em SIAM Journal on Applied Algebra and Geometry}, 3(2):193--230,
  2019.

\bibitem{baltean2018selecting}
Radu Baltean-Lugojan, Pierre Bonami, Ruth Misener, and Andrea Tramontani.
\newblock Selecting cutting planes for quadratic semidefinite
  outer-approximation via trained neural networks.
\newblock 2018.

\bibitem{blekherman2020sparse}
Grigoriy Blekherman, Santanu~S Dey, Marco Molinaro, and Shengding Sun.
\newblock Sparse {PSD} approximation of the {PSD} cone.
\newblock {\em Mathematical Programming}, 2020.

\bibitem{convex_algebraic_geometry}
Grigoriy Blekherman, Pablo~A. Parrilo, and Rekha~R. Thomas.
\newblock {\em Semidefinite Optimization and Convex Algebraic Geometry}.
\newblock Society for Industrial and Applied Mathematics, USA, 2012.

\bibitem{bochnak2013real}
Jacek Bochnak, Michel Coste, and Marie-Fran{\c{c}}oise Roy.
\newblock {\em Real algebraic geometry}, volume~36.
\newblock Springer Science \& Business Media, 2013.

\bibitem{boman2005factor}
Erik~G Boman, Doron Chen, Ojas Parekh, and Sivan Toledo.
\newblock On factor width and symmetric h-matrices.
\newblock {\em Linear algebra and its applications}, 405:239--248, 2005.

\bibitem{Branden2014spec}
Petter Br\"{a}nd\'{e}n.
\newblock Hyperbolicity cones of elementary symmetric polynomials are
  spectrahedral.
\newblock {\em Optimization Letters}, page 1773–1782, 2014.

\bibitem{cavalcanti2016quantum}
Daniel Cavalcanti and Paul Skrzypczyk.
\newblock Quantum steering: a review with focus on semidefinite programming.
\newblock {\em Reports on Progress in Physics}, 80(2):024001, 2016.

\bibitem{DeyWorking}
Santanu~S. Dey, Aleksandr Kazachkov, Andrea Lodi, and Gonzalo Munoz.
\newblock Sparse cutting planes for quadratically-constrained quadratic
  programs.
\newblock 2019.

\bibitem{gaarding1951linear}
Lars G{\aa}rding et~al.
\newblock Linear hyperbolic partial differential equations with constant
  coefficients.
\newblock {\em Acta Mathematica}, 85:1--62, 1951.

\bibitem{gouveia2019sums}
Jo{\~a}o Gouveia, Alexander Kova{\v{c}}ec, and Mina Saee.
\newblock On sums of squares of $ k $-nomials.
\newblock {\em arXiv preprint arXiv:1912.01371}, 2019.

\bibitem{graham1990ramsey}
Ronald~L Graham, Bruce~L Rothschild, and Joel~H Spencer.
\newblock {\em Ramsey theory}, volume~20.
\newblock John Wiley \& Sons, 1990.

\bibitem{guler1997hyperbolic}
Osman G{\"u}ler.
\newblock Hyperbolic polynomials and interior point methods for convex
  programming.
\newblock {\em Mathematics of Operations Research}, 22(2):350--377, 1997.

\bibitem{Horn1985matrix}
Roger Horn and Charles Johnson.
\newblock {\em Matrix analysis}.
\newblock Cambridge University Press, 1985.

\bibitem{hurwitz1891ueber}
Adolf Hurwitz.
\newblock About the comparison of the arithmetic and geometric mean.
\newblock {\em Journal for pure and applied mathematics}, 1891(108):266--268,
  1891.

\bibitem{Mathematica}
Wolfram~Research{,} Inc.
\newblock Mathematica, {V}ersion 12.1.
\newblock Champaign, IL, 2020.

\bibitem{kocuk2016strong}
Burak Kocuk, Santanu~S Dey, and X~Andy Sun.
\newblock Strong socp relaxations for the optimal power flow problem.
\newblock {\em Operations Research}, 64(6):1177--1196, 2016.

\bibitem{kocuk2017matrix}
Burak Kocuk, Santanu~S. Dey, and Xu~A. Sun.
\newblock Matrix minor reformulation and socp-based spatial branch-and-cut
  method for the ac optimal power flow problem.
\newblock {\em arXiv preprint arXiv:1703.03050}, 2017.

\bibitem{Khaz}
Kazhgali Kozhasov.
\newblock On eigenvalues of symmetric matrices with psd principal submatrices.
\newblock {\em arXiv preprint arXiv:2103.15811}, 2021.

\bibitem{MR1411115}
Pierre Lalonde.
\newblock A non-commutative version of {J}acobi's equality on the cofactors of
  a matrix.
\newblock {\em Discrete Math.}, 158(1-3):161--172, 1996.

\bibitem{Oeding_2011}
Luke Oeding.
\newblock Set-theoretic defining equations of the variety of principal minors
  of symmetric matrices.
\newblock {\em Algebra \& Number Theory}, 5(1):75–109, Aug 2011.

\bibitem{Permenter2018}
Frank Permenter and Pablo Parrilo.
\newblock Partial facial reduction: simplified, equivalent sdps via
  approximations of the psd cone.
\newblock {\em Mathematical Programming}, 171(1):1--54, Sep 2018.

\bibitem{qualizza2012linear}
Andrea Qualizza, Pietro Belotti, and Fran{\c{c}}ois Margot.
\newblock Linear programming relaxations of quadratically constrained quadratic
  programs.
\newblock In {\em Mixed Integer Nonlinear Programming}, pages 407--426.
  Springer, 2012.

\bibitem{RenegarDer06}
James Renegar.
\newblock Hyperbolic programs, and their derivative relaxations.
\newblock {\em Foundations of Computational Mathematics}, 6:59--79, 2006.

\bibitem{renegar2004hyperbolic}
James Renegar.
\newblock Hyperbolic programs, and their derivative relaxations.
\newblock {\em Foundations of Computational Mathematics}, 283(6):59--79, 2006.

\bibitem{reznick1987quantitative}
Bruce Reznick.
\newblock A quantitative version of hurwitz’theorem on the
  arithmetic-geometric inequality.
\newblock {\em J. reine angew. Math}, 377(108-112), 1987.

\bibitem{reznick1989forms}
Bruce Reznick.
\newblock Forms derived from the arithmetic-geometric inequality.
\newblock {\em Mathematische Annalen}, 283(3):431--464, 1989.

\bibitem{sanyal2011derivative}
Raman Sanyal.
\newblock On the derivative cones of polyhedral cones, 2011.

\bibitem{sojoudi2014exactness}
Somayeh Sojoudi and Javad Lavaei.
\newblock Exactness of semidefinite relaxations for nonlinear optimization
  problems with underlying graph structure.
\newblock {\em SIAM Journal on Optimization}, 24(4):1746--1778, 2014.

\bibitem{thompson1968principal}
RC~Thompson.
\newblock Principal submatrices v: Some results concerning principal
  submatrices of arbitrary matrices.
\newblock {\em J. Res. Nat. Bur. Standards Sect. B}, 72(2):115--125, 1968.

\bibitem{wang2019polyhedral}
Yuzhu Wang, Akihiro Tanaka, and Akiko Yoshise.
\newblock Polyhedral approximations of the semidefinite cone and their
  applications.
\newblock {\em arXiv preprint arXiv:1905.00166}, 2019.

\bibitem{wolkowicz2012handbook}
Henry Wolkowicz, Romesh Saigal, and Lieven Vandenberghe.
\newblock {\em Handbook of semidefinite programming: theory, algorithms, and
  applications}, volume~27.
\newblock Springer Science \& Business Media, 2012.

\bibitem{zhang2006schur}
Fuzhen Zhang.
\newblock {\em The Schur complement and its applications}, volume~4.
\newblock Springer Science \& Business Media, 2006.

\bibitem{zinchenko2008hyperbolicity}
Yuriy Zinchenko.
\newblock On hyperbolicity cones associated with elementary symmetric
  polynomials.
\newblock {\em Optimization Letters}, 2(3):389--402, 2008.

\end{thebibliography}

\appendix
%\ks{This section was moved from section\ref{sec:eigen42}}. 
\section{Proofs of results in Section \ref{sec:eigen42}}

We first prove the characterization of the eigenvalues of matrices diagonally congruent to $G(4,2)$ in terms of characteristic polynomials.

\begin{proof}(of lemma \ref{lem:good_roots})
Suppose that $p(x) = (x-r_1)(x-r_2)(x-r_3)(x-r_4)$, so that the roots of $p$ are $r_1, r_2, r_3, r_4$. 

Note that the condition that $a_2\ge 0, a_3 \le 0$ is equivalent to the condition that $e_2^4(r_1, r_2, r_3, r_4), e_1^4(r_1, r_2, r_3, r_4) \ge 0$. These inequalities are equivalent to the condition that $(r_1, r_2, r_3, r_4) \in H(e_2^4)$~\cite{guler1997hyperbolic}. Once we know that $(r_1, r_2, r_3, r_4) \in H(e^4_2)$, $a_2 = e_2^4(r_1, r_2, r_3, r_4) = 0$ is equivalent to the condition that $(r_1, r_2, r_3, r_4)$ lies on the boundary of the hyperbolicity cone.

Every $(r_1, r_2, r_3, r_4) \in H(e_2^4)$ has at most 2 negative entries, and if there were exactly 2 negative negative entries, then $r_1r_2r_3r_4 > 0$ (it cannot be the case that there are two negative entries and a zero entry by interlacing). Therefore, the condition that $a_0 < 0$ is equivalent to there being at most 1 negative entry in $(r_1, r_2, r_3, r_4)$.
\end{proof}

%\ss{wouldn't describe $G(4,2)$ as "almost-nonnegative rooted", maybe use another expression}

\begin{proof}(of \ref{lem:almost_nonnegative_roots})

Consider the characteristic polynomial of the matrix $DG(4,2)D$. By definition, it is
\[
    p(\lambda) = \det(D G(4,2) D - I \lambda) = \sum_{i = 0}^4 \sum_{S \subseteq [n], |S| = i}(-1)^i \det\big((DG(4,2)D)|_S\big)\lambda^{4-i}.
\]

Now, note that because $D$ is diagonal, 
\[
    \det\big((DG(4,2)D)|_S\big) = \det(D|_S)^2\det(G(4,2)|_S).
\]
Also, because $G(n,k)$ is symmetric with respect to permutations of the coordinates, $\det(G(n,k)|_S)$ only depends on the size of $S$. So, we have that 
\[
p(\lambda) = \sum_{i = 0}^4  \det(G(4,2)|_{[i]})\sum_{S \subseteq [n], |S| = i}(-1)^i \det\big((D)^2|_S\big)\lambda^{4-i}.
\]
Now, we simply compute
\[
\det(G(4,2)|_{\{1\}}) = 1,
\]
\[
\det(G(4,2)|_{\{1,2\}}) = 0,
\]
\[
\det(G(4,2)|_{\{1,2,3\}}) = -4,
\]
\[
\det(G(4,2)|_{\{1,2,3,4\}}) = -16.
\]

Now, consider the polynomial 
\[
    q(\lambda) = \sum_{i = 0}^4  \sum_{S \subseteq [n], |S| = i}(-1)^i \det\big((D)^2|_S\big)\lambda^i = b_0 + b_1 \lambda + b_2\lambda^2 + b_3 \lambda^3 + \lambda^4.
\]
This is equal to the characteristic polynomial of the matrix $D^2$. As $D^2$ is a diagonal matrix with nonnegative real entries, its eigenvalues are nonnegative. Moreover, if $q$ is a polynomial with nonnegative real roots, then there is a diagonal matrix $D$ so that $q$ is its characteristic polynomial.

Finally, note that from our above characterization of the coefficients of $p$,
\[
    p(\lambda) = -16b_0 + -4b_1 \lambda + b_3 \lambda^3 + \lambda^4 .
\]

On the other hand, if $p$ has almost-nonnegative roots, then we can construct the desired $D$ from the roots of $q$, and then $DG(n,k)D$ will have the desired eigenvalues.
\end{proof}

We now prove a number of polynomial inequalities which are equivalent to the good-rooted and almost-real-rooted conditions.
\begin{lemma}\label{lmma:good_inequalities}
$ p = a_0 + a_1 x - x^3 + x^4$ has good roots if and only if $a_0 < 0$, and 
\[
-4 a_1^3 - 27 a_1^4 - 6 a_1^2 a_0 - 27 a_0^2 - 192 a_1 a_0^2 + 256 a_0^3 \ge 0.
\]
\end{lemma}
\begin{proof}
This polynomial $-4 a_1^3 - 27 a_1^4 - 6 a_1^2 a_0 - 27 a_0^2 - 192 a_1 a_0^2 + 256 a_0^3$ is the discriminant of $p$, which is nonnegative if and only if the number of real roots of $p$ is a multiple of $4$, or $p$ has a double root.

If $p$ has 4 nonreal roots, say $r_1, r_2, r_3, r_4$, then they must come in conjugate pairs, so that, say, $r_1 = \bar{r_2}$ and $r_3 = \bar{r_4}$, which would imply that then 
\[
a_0 = r_1r_2r_3r_4 = \change{|r_1|^2|r_3|^2},
\]
is nonnegative, a contradiction.

Similarly, if $p$ has a double root, say $r_3 = r_4$, and a pair of complex conjugate roots, say $r_1 = \bar{r_2}$ then we see that 
\[
    a_0 = r_3^2\change{|r_1|^2} \ge 0,
\]
which is a contradiction.
\end{proof} 

\begin{lemma}\label{lmma:almost_inequalities}
$ p = a_0 + a_1 x - x^3 + x^4$ has almost-nonnegative roots if and only if $a_0 < 0$, $a_1 < 0$, and there is $k > 0$, so that \change{the }following 4 inequalities are satisfied:

\begin{align*}
    %Equation 1
    \Big(-\frac{27 a_1^4}{256}-\frac{9 a_1^3 k}{32}+\frac{a_1^3}{16}-\frac{9}{16} a_1^2 a_0 k+\frac{3 a_1^2 a_0}{128}-\frac{a_1^2 k^3}{4}+\frac{a_1^2 k^2}{16}+\frac{3 a_1 a_0^2}{16}\\
    -\frac{5}{4} a_1 a_0 k^2+\frac{9 a_1 a_0 k}{32}-\frac{a_0^3}{16}-\frac{a_0^2 k^2}{2}+\frac{9 a_0^2 k}{16}
    -\frac{27 a_0^2}{256}-a_0 k^4+\frac{a_0 k^3}{4}\Big)\ge 0,\\\\
    %Equation 2
    \Big(\frac{27 a_1^4}{256}+\frac{9 a_1^3 k}{32}+\frac{a_1^3}{8}+\frac{45}{128} a_1^2 a_0 k-\frac{9 a_1^2 a_0}{128}+\frac{a_1^2 k^3}{4}+
    \frac{45 a_1^2 k^2}{16}+a_1^2 k+\\
    \frac{37 a_1^2}{8}-\frac{9 a_1 a_0^2}{128}+\frac{11}{16} a_1 a_0 k^2-\frac{43 a_1 a_0 k}{32}-\frac{53 a_1 a_0}{32}+9 a_1 k^3+\frac{27 a_1 k}{2}-\\
    3 a_1+\frac{a_0^3}{64}+\frac{3 a_0^2 k^2}{16}+ \frac{23 a_0^2 k}{128}+\frac{77 a_0^2}{256}+\frac{a_0 k^4}{2}+\frac{27 a_0 k^3}{8}-
    3 a_0 k^2+ \frac{3 a_0 k}{8}-\\
    \frac{3 a_0}{8}+8 k^5-2 k^4+16 k^3-4 k^2\Big)\le 0,\\\\
    % Equation 3
    \Big(-\frac{3 a_1^3}{16}+\frac{3 a_1^2 a_0}{64}-\frac{19 a_1^2 k^2}{16}-a_1^2 k-\frac{17 a_1^2}{16}+ \frac{11 a_1 a_0 k}{32}+\frac{17 a_1 a_0}{32}-\\
    \frac{9 a_1 k^3}{2}-a_1 k^2+3 a_1 k-\frac{15 a_1}{2}-
    \frac{3 a_0^2 k}{32}-\frac{17 a_0^2}{256}-\frac{5 a_0 k^3}{4}+\frac{9 a_0 k^2}{8}-\\
    \frac{11 a_0 k}{8}+\frac{21 a_0}{8}-4 k^5+k^4-16 k^3+33 k^2-38 k+9\Big)\ge 0,\\\\
    % Equation 4
    \left(-\frac{3 a_1^2}{16}-3 a_1 k+\frac{5 a_1}{2}+\frac{3 a_0 k}{8}-\frac{5 a_0}{8}+2 k^3-11 k^2+12 k-7\right)\le 0.\\
\end{align*}
\end{lemma}
\begin{proof}
The classical results that we need about real rooted univariate polynomials, such as the Newton identities and the Hermite-Sylvester conditions can be found at \cite[Section 3.1]{convex_algebraic_geometry}.

If we have the sign conditions on the coefficients, $a_0 < 0$, $a_1 < 0$, $k > 0$, then the polynomial 
$q = \frac{a_0}{-16} + \frac{a_1}{-4}x + kx^2 - x^3+x^4$ has coefficients which alternate in sign. If $q$ is real rooted, then we can apply Descartes' rule of signs to conclude that $q$ has nonnegative real roots.

The remaining inequalities cut out the space of real-rooted polynomials. This follows from the Hermite-Sylvester criterion for the polynomial having real roots. It states that if we let $m_k = \sum_{i=1}^4 r_i^k$, where $r_1, r_2, r_3, r_4$ are the roots of $q$, then $p$ has nonnegative real roots if and only if the $4\times 4$ matrix $M$ given by
\[
M_{ij} = m_{i+j}.
\]
is positive semidefinite.

We can then use the Newton identities
to determine the $m_k$ in terms of $a_0$, $a_1$ and $k$.

Once $M$ has been computed, the 4 polynomials above are the 4 coefficients of the characteristic polynomial of $M$. $M$ being positive semidefinite is equivalent to these 4 polynomials alternating in sign, which results in the four inequalities listed.
\end{proof} 

\begin{proof}(Of lemma \ref{lem:good_implies_almost_nonnegative})

We want to show that for all $a_0$ and $a_1$ satisfying the conditions of \ref{lmma:good_inequalities}, there exists $k$ satisfying the conditions of \ref{lmma:almost_inequalities}.

We are now at the point where we can directly apply any quantifier elimination algorithm to solve this problem, say the one included in Mathematica\cite{Mathematica}.
The results of this computation show that the lemma holds.
\end{proof}

\end{document}